\documentclass[11pt]{amsart}

\usepackage{todonotes}
\newcommand{\RvdH}[1]{\todo[inline, color=magenta]{{\rm Remco: #1}}}
\newcommand{\ch}[1]{{\color{red}#1\color{black}}}

\usepackage{amsmath,amssymb,amsthm,bbm} 
\usepackage{hyperref} 
\usepackage{geometry} 
\usepackage{xcolor} 
\usepackage{tikz} 
\usepackage{xfp} 

\definecolor{theme}{RGB}{15, 87, 24} 
\definecolor{lighttheme}{RGB}{51, 153, 63} 
\definecolor{block}{RGB}{102, 60, 0} 
\definecolor{lightblock}{RGB}{255, 238, 214} 
\definecolor{alert}{RGB}{212, 43, 43} 

\def\pagespace{3cm}
\hypersetup{
    colorlinks=true,
    linkcolor=theme,
    citecolor=lighttheme
}

\newtheorem{theorem}{Theorem}[section]
\newtheorem{proposition}[theorem]{Proposition}
\newtheorem{lemma}[theorem]{Lemma}
\newtheorem{corollary}[theorem]{Corollary}

\newtheorem{assumption}{Assumption}

\newcommand{\collaborator}[3]{\marginpar{\footnotesize\textcolor{#2!50!black}{\textbf{#1:}} {#3}}}
\newcommand{\benoit}[1]{\collaborator{Beno\^it}{green}{#1}}

\newcommand{\pskip}{\left.\hspace{0.5cm}\right.} 

\newcommand{\bR}{\mathbb{R}}
\newcommand{\prob}{\mathbb{P}}
\newcommand{\expec}{\mathbb{E}}

\newcommand{\rI}{\mathbbm{1}}
\newcommand{\wG}{\mathcal{G}}
\newcommand{\rG}{\mathbb{G}}
\newcommand{\rH}{\mathbb{H}}
\newcommand{\uniform}{\textsc{Uniform}([0,1])}
\newcommand{\lwc}{\overset{\mathrm{lwc}}{\longrightarrow}}
\newcommand{\lcp}{\overset{\mathrm{lc}\prob}{\longrightarrow}}
\newcommand{\lpc}{\overset{\mathrm{lpc}}{\longrightarrow}}
\newcommand{\MSF}{\mathrm{F}}
\newcommand{\rMSF}{\mathbb{F}}
\newcommand{\MST}{\mathrm{T}}
\newcommand{\rMST}{\mathbb{T}}
\renewcommand{\root}{o}

\newcommand{\prim}{\mathcal{P}}

\newcommand{\comp}{\mathcal{C}}
\newcommand{\probc}{\overset{\prob}{\longrightarrow}}
\newcommand{\speed}{\mathcal{K}}

\usepackage{longtable}
\usepackage[toc,page]{appendix}
\newcommand{\eqn}[1]{\begin{equation}#1\end{equation}}

\newcommand{\sss}{\scriptscriptstyle}
\newcommand{\convp}{\xrightarrow{\sss\prob}}
\newcommand{\invisible}[1]{}

\begin{document}

\author{Beno\^it Corsini \and Rowel G\"undlach \and Remco van der Hofstad}
\title{Local limit of Prim's algorithm}

\setcounter{tocdepth}{1}

\begin{abstract}
{We study the local evolution of Prim's algorithm on large finite weighted graphs. When performed for $n$ steps, where $n$ is the size of the graph, Prim's algorithm will construct the minimal spanning tree (MST). We assume that our graphs converge locally in probability to some limiting rooted graph. In that case, Aldous and Steele already proved that the local limit of the MST converges to a limiting object, which can be thought of as the MST on the limiting infinite rooted graph.

Our aim is to investigate {\em how} the local limit of the MST is reached \textit{dynamically}. For this, we take $tn+o(n)$ steps of Prim, for $t\in[0,1]$, and, under some reasonable assumptions, show how the local structure interpolates between performing Prim's algorithm on the local limit when $t=0$, to the full local limit of the MST for $t=1$. Our proof relies on the use of the recently developed theory of {\em dynamic local convergence}. We further present several examples for which our assumptions, and thus our results, apply.
}
\end{abstract}

\maketitle

\tableofcontents

\section{Introduction}\label{sec:intro}

\subsection{Background and novelty}

In this paper we analyse the evolution of the local structure of the subtree discovered by Prim's algorithm~\cite{Prim1957} as the number of steps in the algorithm varies.
This process is one of the most commonly used greedy algorithms to find the minimal spanning tree (MST) of a rooted weighted graphs.
However, Prim's algorithm is different  from other greedy methods to construct the MST, such as Bor\r{u}vka's algorithm \cite{Nesetril2001} and Kruskal's algorithm \cite{Kruskal1956}, as it discovers the tree {\em locally} starting from a given root. More precisely, it sequentially adds the neighbouring edge with lowest weight that does not form a cycle with the current component. When running Prim's algorithm to completion, the outcome is known to be the MST of the whole finite graph.
In that case, the local limit has already been analysed in~\cite[Theorem 5.6]{Aldous2004}, where the authors show local convergence towards the minimal spanning forest of the infinite graph.
\smallskip

The {\em novelty} of our paper is the study of the outcome of Prim's algorithm when only considering a partial number of steps.
More precisely, we provide a characterisation of the limit of Prim's algorithm as a function of $t$ when it runs for $tn+o(n)$ steps on a graph with $n$ vertices, for $t\in[0,1]$.
This provides a process on rooted weighted graphs and fits within the recent research field of dynamic local convergence~\cite{Dort2023,Milewska2023}.
This process convergence provides key insights on the evolution of the structure of Prim's algorithm when the number of steps grows.
\smallskip

One of the noteworthy interpretations of our main result (Theorem~\ref{thm:Prim convergence}), is {\em how} Prim's algorithm discovers the local neighbourhood of the MST around the root.
Indeed, our result states that it first spends a sub-linear time in the local neighbourhood, during which it creates the \textit{invasion percolation cluster}, after which it explores far away parts of the graph, until it reaches a linear number of steps, where it regularly returns to the local neighbourhood and adds edges of the minimum spanning forest of larger and larger weights.
Our result further shows that the local structure of the graph does not change as soon as we run Prim's algorithm for $n-o(n)$ steps, allowing for a sub-linear correction.
In particular, this means that the last $o(n)$ steps of the algorithm always explore faraway neighbourhoods of the graph.

\subsection{Main result}\label{sec:main result}

For the rest of this work, we consider a sequence $(G_n)_{n\geq1}$ of graphs on $n$ vertices, which converges locally in probability towards a rooted graph $(G,o)$. Local convergence was defined independently by Benjamini and Schramm \cite{Benjamini2001} and Aldous and Steele \cite{Aldous2004}; see \cite[Chapter 2]{VanderHofstad2024} for an extensive introduction to local convergence.
We further endow $(G_n)_{n\geq1}$ and $G$ with independent $\uniform$ edge weights and sample a random vertex $o_n$ from $G_n$;
we write $\rG=(G,o,w)$ and $\rG_n=(G_n,o_n,w_n)$ for the corresponding rooted weighted graphs and call these the \textit{standard extensions} of $(G_n)_{n\geq1}$ and $(G,o)$.
We are now interested in the local limit of $\prim_{k_n}(\rG_n)$, defined as the subtree obtained after $k_n$ steps of Prim's algorithm, when $n\rightarrow\infty$ and for arbitrary increasing sequences $(k_n)_{n\geq1}$.
\smallskip

As discussed above, when $k_n=n$, then $\prim_{k_n}(\rG_n)$ is exactly the MST of $(G_n,w_n)$, which converges locally towards the minimum spanning forest $\rMSF^{\rG}$ of $\rG$~\cite[Theorem~5.4]{Aldous2004} (see Section~\ref{sec:msf + eipc} for exact definitions).
Using this result, we know that there exist growing radii $(r_n)_{n\geq1}$ such that $(\rG_n,\prim_n(\rG_n))$ can asymptotically be coupled with $(\rG,\rMSF^{\rG})$ on balls of radius $r_n$. 
\smallskip

We say that the sequence of functions $(k_n(t))_{n\geq1,t\in[0,1]}$ is a \textit{linearly growing sequence} with respect to $(G_n)_{n\geq1}$ if $k_n(0)=r_n$ and $k_n(t)/n\rightarrow t$ as $n\rightarrow\infty$ for all $t\in[0,1]$ (see Section~\ref{sec:local prop} for a more precise definition of linearly growing sequences).
Our main result characterises the behaviour of Prim's algorithm for any arbitrary linearly growing sequence.
In particular, we investigate the convergence of the rooted graph obtained by applying Prim $k_n(t)$ steps as a stochastic process on rooted graphs, using the recently developed theory of dynamic local convergence \cite{Milewska2023}.
\smallskip

On the infinite graph $\rG$, we note that the minimum spanning forest $\rMSF^{\rG}$ contains the subtree $\prim_k(\rG)$ obtained by running Prim's algorithm on $\rG=(G,o,w)$ for any arbitrarily chosen number of steps $k$.
Thus, $\rMSF^{\rG}$ also contains the limit $\prim_\infty(\rG)$ of Prim's algorithm, corresponding to the invasion percolation cluster.
We now generalise this definition to the \textit{expanded invasion percolation cluster} at level $p$, which we denote by  $\rMSF^{\rG}_+(p)$. We define this as the union of the invasion percolation cluster and the edges in the minimal spanning forest $\rMSF^{\rG}$ with weight less than $p$ (see Section~\ref{sec:msf + eipc} for a more precise definition).
This tree is fundamental in characterising the limit of Prim's algorithm and we now provide the two assumptions required for our main theorem to hold.
\medskip

Given a rooted weighted graph $\rG=(G,o,w)$, write $\comp^{\rG}_o(p)$ for the component containing $o$ when only keeping edges with weight less than $p$ and let $\theta(p)=\theta^{\rG}(p)$ be the (annealed) survival probability of $o$ at level $p$ in $\rG$, i.e.,
\begin{align}\label{eq:theta}
    \theta(p)&=\prob\Big(\big|\comp^{\rG}_o(p)\big|=\infty\Big)\,.
\end{align}
We further define its critical value by $p_c=\inf\{p:\theta(p)>0\}$ and its (pseudo-)inverse, for $t\in[0,1)$, by
\begin{align}\label{eq:theta inv}
    \theta^{-1}(t)&=\inf\Big\{p:\theta(p)>t\Big\}\,.
\end{align}
For convenience, we define $\theta^{-1}(1)=1$. When $\theta(p)>0$, it is possible for the root of $\rG$ to be in an infinite component at level $p$.
It is thus natural to wonder what the rest of the graph looks like in the case where $\theta(p)>0$ but the root does not belong to an infinite component.
Our first assumption states that, whenever $\theta(p)>0$, there exists a component in $\rMSF^{\rG}$ which is infinite:

\begin{assumption}[Percolation function and infinite percolation components]\label{ass:speed}
    ~The (infinite)\\
    rooted~and~weighted graph $\rG$ is said to satisfy the {\rm local percolation of a giant} assumption if, for any $p$, the following equivalence holds:
    \begin{align*}
        \theta(p)>0~\Longleftrightarrow~\prob\Big(\exists v\in V(\rMSF^{\rG})\colon \big|\comp^{\rG}_v(p)\big|=\infty\Big)=1\,.
    \end{align*}
\end{assumption}
Assumption \ref{ass:speed} entails that even when the root does not percolate at level $p$, if $p$ is such that $\theta(p)>0$, then there are vertices close to the root that percolate.
It plays a key role in ensuring that the giant component in the finite graph is locally visible, and can thus be reached by Prim's algorithm after a bounded number of steps.
\smallskip

In Section~\ref{sec:examples}, we discuss why this assumption is required and provide examples of (random) graphs that do not satisfy it.
We further conjecture that this assumption can be omitted and the main result would remain true, since we did not find counterexamples to prove otherwise.
\smallskip

Besides the relation between the probability to have an infinite component and the existence of an  infinite component in the minimum spanning forest, we need a further assumption on $\theta$.
Indeed, since $\theta$ plays a key role in the evolution of the limit (see Theorem~\ref{thm:Prim convergence}), it is also important for this function to be smooth enough, as stated in the following assumption:

\begin{assumption}[Smoothness of the percolation function]\label{ass:smooth}
    The (infinite) rooted weighted graph $\rG$ is said to satisfy the {\rm smooth percolation} assumption if $p\mapsto\theta(p)$, considered as a function on  $(p_c,1]$, is continuous and strictly monotone.
\end{assumption}
We emphasize that $\rG$ is possibly random, and that $\theta(p)$ is the {\em annealed} percolation probability, where we take the average w.r.t.\ the random graph $\rG$ and the edge statuses on it. Note that Assumption \ref{ass:smooth} implies that $t\mapsto \theta^{-1}(t)$ is continuous and strictly increasing as a function from $[0,1]$ to $[p_c,1]$, with $\theta^{-1}(0)=p_c$, which, in fact, is the main property that $\theta$ needs to satisfy for our proofs. 
\smallskip

While we believe Assumption \ref{ass:smooth} holds in most cases, it is still possible to create counterexamples where $\theta$ jumps multiple times (see Section~\ref{sec:examples}).
However, these examples tend to also not satisfy Assumption \ref{ass:local} below, and so it is unclear to us whether Assumption \ref{ass:smooth} is redundant or not. Without proof of the contrary, we include it and keep its removal as an open problem.
\smallskip

The previous two assumptions only pose  conditions on the local limit of the graph sequence. We next state an assumption to make sure that the sequence itself behaves as desired. Given the relation between the MST  and connected components at level $p$ in the percolated graph, we next state an assumption linking the survival probability and size of the giant, as introduced in~\cite[Theorem 2.28]{VanderHofstad2024} for general (unpercolated) graph sequences:

\begin{assumption}[Percolation giant is almost local]\label{ass:local}
    The sequence of rooted weighted graphs $(\rG_n)_{n\geq1}$ is said to satisfy the {\rm percolation giant is almost local} assumption if, for any $p$,
    \begin{align*}
        \lim_{k\rightarrow\infty}\limsup_{n\rightarrow\infty}\frac{1}{n^2}\expec\bigg[\Big|\Big\{u,v\in V(\rG_n):\big|\comp^{\rG_n}_u(p)\big|\geq k,\big|\comp^{\rG_n}_v(p)\big|\geq k,\comp^{\rG_n}_u(p)\cap\comp^{\rG_n}_v(p)=\varnothing\Big\}\Big|\bigg]=0\,.
    \end{align*}
\end{assumption}
Assumption \ref{ass:local} allows us to relate the size of the largest component in $\rG_n$ at level $p$ to $\theta^{\rG}(p)$ and we invite the reader to take a look at Propositions~\ref{prop:theta} and~\ref{prop:local giant} to gain a first idea of the importance of this assumption.
This assumption plays an essential role in our main theorem, as it is a key property required to connect the behaviour of percolated components between finite and infinite graphs
\medskip

With these assumptions in hand, we now have all the ingredients to state our main theorem.
The following result relies on the \textit{local process convergence}, which corresponds to the natural extension of the local topology on graphs to the process topology (see Section~\ref{sec:lpc} for more details):

\begin{theorem}[Local limit of Prim's algorithm]\label{thm:Prim convergence}
    Let $(G_n)_{n\geq1}$ be converging locally in probability to $(G,o)$.
    Let $(\rG_n)_{n\geq1}$ and $\rG$ be the respective standard expansions of $(G_n)_{n\geq1}$ and $(G,o)$.
    Define $\theta$ as in~\eqref{eq:theta} and let $k=(k_n(\cdot))_{n\geq1}$ be a linearly growing sequence with respect to $(G_n)_{n\geq1}$ as defined in Section~\ref{sec:local prop}.
    If $\rG$ satisfies Assumption~\ref{ass:speed}, $\theta(p)$ satisfies Assumption~\ref{ass:smooth}, and $(\rG_n)_{n\geq1}$ satisfies Assumption~\ref{ass:local}, then
    \begin{align*}
        \Big(\prim_{k_n(t)}(\rG_n)\Big)_{t\in[0,1]}\lpc\Big(\rMSF^{\rG}_+\big(\theta^{-1}(t)\big)\Big)_{t\in[0,1]}\,,
    \end{align*}
    where the convergence occurs according to the local process convergence defined in Section~\ref{sec:lpc}.
\end{theorem}

The proof of Theorem~\ref{thm:Prim convergence} can be found in Section~\ref{sec:process}.
It can be reduced to three main steps.
First, we prove the one-dimensional convergence of Prim's algorithm towards the correct limit.
This proof boils down to the fact that the largest component in $\rG_n(p)$ has size approximately $n\theta(p)$, and so Prim's algorithm applied $n\theta(p)$ steps will first quickly reach the largest component and then remain in that component, explaining the structure of $\rMSF^{\rG}_+(p)$.
Second, we prove that the previous one-dimensional convergence can be extended to the finite-dimensional distributions, and show that the limit satisfies the assumptions of the local process convergence as defined in Section~\ref{sec:lpc}. 
Finally, we prove the tightness of the process (as defined in Section~\ref{sec:lpc}), and complete the proof of Theorem~\ref{thm:Prim convergence}, in Section~\ref{sec:process}.
\smallskip

It is worth mentioning that the previous convergence not only occurs as a process convergence but also jointly with the convergence of $\rG_n$ towards $\rG$.
This is a direct consequence of the proof method from Section~\ref{sec:thm}, where we start by coupling $\rG_n$ and the minimal spanning tree on it, to $\rG$ and the minimal spanning forest on it, and show the convergence of Theorem~\ref{thm:Prim convergence} under this coupling.
We decided not to state it as a joint convergence here for clarity.
\smallskip

\paragraph{\bf Organisation}This paper is organised as follows.
In Section~\ref{sec:background} we give some background on the different definitions from the introduction and state some useful preliminary results.
In Section~\ref{sec:thm} we then apply the previous results in order to prove Theorem~\ref{thm:Prim convergence}.
Finally, in Section~\ref{sec:examples}, we consider a few interesting specific cases and examples.
We also invite the reader to take a look at Appendix~\ref{app:notations} for a list of often-used symbols and notations.
We now provide a more detailed review of the literature relevant to this work.

\subsection{Related literature}\label{sec:literature}
The introduction of local convergence \cite{Aldous2004,Benjamini2001} at the start of this millennium has formalised the connection between analysing graphs and graph processes on different scales. The power of this technique is demonstrated in our results for Prim's algorithm, but has also been demonstrated for many different processes in the literature. We present an overview of how the introduction of local convergence gave rise to a formal local description of graphs and how this was used to locally analyse global graph properties and graph processes. Next, we emphasise the importance of the MST and present the state-of-the-art results for the MST on random graphs, both from a local and global perspective.
\medskip

\paragraph{\bf Local convergence of graphs}
The early applications of local convergence show locally tree-like behaviour for popular sparse random graphs. Key examples include
the Erd\H{o}s-R\'enyi graph \cite{Dembo2010}, the uniform random graph and the configuration model \cite{Dembo2010}, inhomogeneous  random graph \cite{Bollobas2007} and \cite[Chapter 3]{VanderHofstad2024}, and the preferential attachment model \cite{Berger2014,Garavaglia2023}. 
Recent years have seen an increasing interest in this line of work and have extended the results to more complex graphs such as uniform planar graphs \cite{Stufler2021}, spatial inhomogeneous  random graphs \cite{Hofstad2023}, 
Gibbs random graphs \cite{Endo2020}, and random intersection graphs \cite{Kurauskas2022}.
For more extensive overviews, we refer to \cite{Banarjee2024,VanderHofstad2024}.
While the field of local convergence for random graphs developed, others investigated how to optimally make use of these novel results. Let us mention two main applications of local convergence results: First, local convergence proved itself to be useful in analysing global graph properties, such as the existence, size and uniqueness of the giant. Second, processes on random graphs can often be analysed from a local perspective. Next, we give a short overview of both applications.
\medskip

\paragraph{\bf Local analysis of global objects}
Many sparse random graphs are characterised by their local limit to such a strong degree, that their structure can be used to analyse global graph properties. 
\cite[Chapter 2.5]{Hofstad2023} illustrates this in detail for different graphs measures, such as
the global cluster coefficient and  assortativity coefficient for graphs that have a uniformly integrable second and third moment for the degree of a uniformly chosen vertex respectively. Related is the PageRank distribution \cite{Garavaglia2020} that is also well studied from a local perspective.

Such results also extend to an analysis of the number of spanning trees for large sparse graphs \cite{Lyons2005,Salez2013}, to the spectral measure for locally tree-like graphs \cite{Bhamidi2012,Bordenave2011} and to graph colourings \cite{Coja-Oghlan2018}.
A similar analysis for the critical percolation probability for transitive graphs has seen much attention over the years. First conjectured that this was indeed a local property (see \cite{Benjamini2011} for details), many works have since been extending the class of graphs for which the critical percolation probability was indeed local, see \cite{Contreras2023} and the references therein.
The conjecture was finally proven in full generality in \cite{Easo2023}.  
Finally, recent work shows that, under some conditions on the graph, the giant component can be completely characterised from a local perspective \cite{VanderHofstad2021b,VanderHofstad2024, VanderHofstad2024b}. In this case, the giant was also described as \lq almost local\rq.
 
We finally remark that these results find direct practical applications in the mathematical physics literature via, for example, the factor model \cite{Dembo2013} and the Ising model. The latter model has received a significant amount of attention. Here 
one assigns labels (or spins) $\{-1,1\}$ to a vertices of a random graph according to the Boltzmann distribution, which forms a central representation for ferromagnetic fields. The works of \cite{Dembo2010,Dembo2013, Montanari2012} show that for a multitude of locally tree-like graphs the spin distributions can be analysed from the local limit.
These works have later been extended to different graphs and to different properties of the Ising model from both a theoretical and practical point of view, see \cite{Albenque2021,Anandkumar2012,Dembo2013, Dommers2014,Yu2023}.
\medskip

\paragraph{\bf Processes on a random graph}
A second key development initiated by local convergence is the translation of processes on graphs to their local limit. Ideally
one hopes that analysing a process on a random graph is similar to analysing the process on the local limit. We first observe such a phenomenon on dynamic random graphs, for example the Erd\H{o}s-R\'enyi and inhomogeneous  random graph \cite{Dort2023} and the random intersection graph \cite{Milewska2023}.   
Interacting particle systems, where states of vertices develop via a Markov or diffusion process based on their neighbours, have been extensively analysed, see \cite{Lacker2023} and the references therein, where the authors show that the vertex interaction translates to the local limit. Related examples have also been studied, such as the contact process \cite{Mountford2013}, 
SIR epidemic models \cite{Milewska2025}, and voter models \cite{Huo2018}. 
\medskip

\paragraph{\bf Algorithms for the minimum spanning tree}
There are many classical construction algorithms for the MST dating back to the early 20th century, such as Bor\r{u}vka's algorithm \cite{Nesetril2001}, Kruskal's algorithm \cite{Kruskal1956} and Prim's algorithm \cite{Prim1957}, which generally run in $O(E\log(V))$ time complexity \cite{Pettie2002}. From a discrete optimisation point of view, the main interest in the MST is to improve the efficiency of search algorithms, see for example \cite{Khan2019} and the reference therein.

For this work, the geometry of the MST in disordered networks is of significant importance. To the best of our knowledge, interest in the topic arose in the mid-eighties for properties of MST on the complete graph with random edge weights \cite{Frieze1985} and for the Euclidean minimum spanning tree \cite{Aldous1992,Steele1988}
and the references therein. This was later further generalised in \cite{Alexander1995}.
Shortly after the introduction of local weak convergence in the beginning of the new millennium,
Aldous and Steele \cite{Steele2002} examined the MST on the weighted complete graph, and related that to a minimum spanning forest on the Poisson weighted infinite tree (PWIT). This work was later extended to general graphs in \cite{Aldous2004}.

The geometry and other properties of the MST find many applications in applied physics, see for example \cite{Dobrin2001} and the references therein. 
A classic example of such geometrical property is the diameter of the MST, often used as an upper bound for the typical distance between uniform vertices. These results hold for many classes of disordered networks \cite{Braunstein2003}, 
due to the universality of the MST \cite{Dobrin2001}.
\medskip

\paragraph{\bf Local limits for the minimum spanning tree} So far, local convergence of the MST (i.e., $k_n=n$) has been shown for general graphs, see \cite[Theorem 5.4]{Aldous2004} and more in depth for the complete graph with random edge weights \cite{Addario-Berry2013,NacTan22}. 
\medskip

\paragraph{\bf Global limits for the minimum spanning tree}  The global perspective of the MST is usually captured by its {\em scaling limit}. This was first derived for the complete graph in \cite{Addario-Berry2017}, and is closely related to the Erd\H{o}s-R\'enyi random graph. For other random graph models, the global geometry of the MST has also been analysed before, such as the 3-regular graph \cite{Addario-Berry2021} and the inhomogeneous random graph with power-law degrees \cite{Bhamidi2024}. Interestingly, while the scaling limit of these MSTs are intimately related to large critical percolation cluster on them, the Hausdorff dimension of the MST is generally larger than that of critical percolation clusters. Indeed, for the complete graph, the Hausdorff dimension of critical percolation clusters equals 2 (as for the closely related continuum random tree), while for the scaling limit of the MST it is 3 \cite{Addario-Berry2017}.
\medskip

\paragraph{\bf Quenched versus annealed}
As one can consider this model as random disorder (the uniform edge weights) in a random environment (the random graph), we emphasise that our results are for the \textit{annealed setting}, that is, where one averages out over the random environment. This is in contrast with the \textit{quenched setting}, where one shows results for an almost sure realisation of the environment. This distinction is important as it is known that results can differ between regimes, see for example \cite{Arous2005} and the references therein.

\subsection{Discussion of the results}
In the following, we discuss the implications of Theorem~\ref{thm:Prim convergence}. We first illustrate how this result completely describes how Prim's algorithm discovers a random graph from a local perspective. We provide simulations of Prim's algorithm on large graphs and show how it follows patterns found in our theoretical results. Finally, we show that our results also provide a good description of the addition times of edges.
\medskip

\paragraph{\bf Discovering the MST}

Analysing the local limit of the MST on (random) graphs via Prim's algorithm involves two limits: one relating to the graph size (i.e., $n$) and the other relating to number of steps of Prim's algorithm (i.e., $k$). Except for the clear condition $n\geq k$, there is some freedom in how these limits are taken. Theorem \ref{thm:Prim convergence} illustrates the behaviour for different choices of the relation between $n$ and $k$. 
First, by taking $k=o(n)$, we locally obtain the invasion percolation cluster of the local limit $\rMSF^{\rG}_+(p_c)=\prim_\infty(\rG)$.
This agrees with simply taking these limits \textit{sequentially}, i.e., taking $n\to\infty$ followed by $k\to\infty$. 
\smallskip

While clearly the local neighbourhood of the MST is not fully explored in this setting (consider the example of a single vertex attached to the root with a very high weight), this directly implies that $k$ growing (arbitrary slowly) in $n$ is not sufficient to discover the full local neighbourhood of the MST (at least when $p_c<1$).
Indeed, for some small $\varepsilon>0$, at time $\varepsilon t$, by Proposition \ref{prop:main result}, Prim's algorithm is discovering the largest percolation component of the graph percolated at $\theta^{-1}(\varepsilon)>p_c$, which is located outside of the local neighbourhood. By Proposition \ref{prop:local giant}, the largest component is of linear size when percolated at $p>p_c$, we find that only after a linear number of steps, Prim's algorithm is able to {\em return} to the local neighbourhood. 
Indeed, our results state that, in the time interval $[cn, (c+\delta)n]$, Prim's extends its search in the local neighbourhood by connecting edges with weights in between $\theta^{-1}(c)$ and $\theta^{-1}(c+\delta)$.
\smallskip

The local limit of the MST is therefore constructed in two \textit{phases}. In the \textit{first phase}, consisting of $o(n)$ steps, Prim's algorithm basically explores the invasion percolation cluster on the local limit. In the \textit{second phase}, after spending $\varepsilon n$ steps away from the root, it returns to the neighbourhood of the root, and adds edges and groups of vertices that belong to the MST.
This was also noted in \cite{Addario-Berry2013} for the complete graph.
\medskip

\paragraph{\bf Simulations}

Our results imply that there are vertices in the local neighbourhood that are explored very early by Prim's algorithm (first phase) and some that are explored very late (second phase). Such behaviour implies that it takes around $n(1-o(1))$ steps for Prim's algorithm to fully discover the local limit of the MST, as we confirm in a  simulation study.

In Figures~\ref{fig:square}, \ref{fig:triangle}, and \ref{fig:regular} we ran Prim's algorithm on large grids, triangular lattices, and $3$-regular graphs respectively.
While all figures show about 1000 vertices of the graph, the simulations were run on a larger graph extending past the boundary of the images, in order to better reproduce the effect of a {\em local limit}.
In all cases, we can see that some edges close to the root are added late in the algorithm (i.e., they are red).
Finally, in Figure~\ref{fig:grid}, we provide a representation on a much larger scale with a $2000\times2000$ grid; in this case, we once again see that late components, in red, are added in every neighbourhood around the root. 

\begin{figure}[htb]
    \centering
    \includegraphics[width=0.25\linewidth]{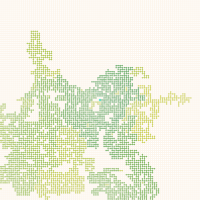}
    \includegraphics[width=0.25\linewidth]{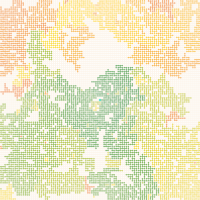}
    \includegraphics[width=0.25\linewidth]{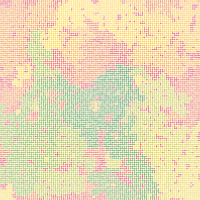}
    \caption{Simulation of Prim's algorithm on a large grid, starting from the root (in turquoise, in the middle). Vertices and edges are coloured based on the step at which Prim's algorithm adds them to the current tree (from green to yellow, and finally red) and the three images correspond to running Prim's algorithm for $n/3$, $2n/3$, and $n$ steps.}
    \label{fig:square}
\end{figure}

\begin{figure}[htb]
    \centering
    \includegraphics[width=0.25\linewidth]{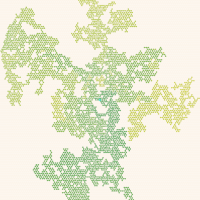}
    \includegraphics[width=0.25\linewidth]{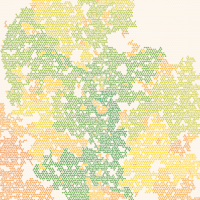}
    \includegraphics[width=0.25\linewidth]{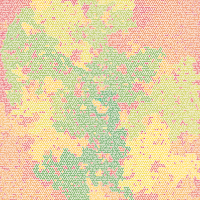}
    \caption{Simulation of Prim's algorithm on a large triangular lattice, starting from the root (in turquoise, in the middle). Vertices and edges are coloured based on the step at which Prim's algorithm adds them to the current tree (from green to yellow, and finally red) and the three images correspond to running Prim's algorithm for $n/3$, $2n/3$, and $n$ steps.}
    \label{fig:triangle}
\end{figure}

\begin{figure}[htb]
    \centering
    \includegraphics[width=0.25\linewidth]{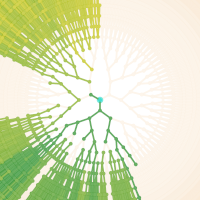}
    \includegraphics[width=0.25\linewidth]{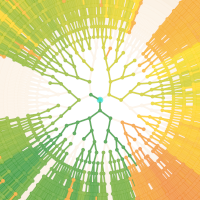}
    \includegraphics[width=0.25\linewidth]{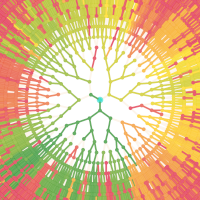}
    \caption{Simulation of Prim's algorithm on a large $3$-regular graph, starting from the root (in turquoise, in the middle). Vertices and edges are coloured based on the step at which Prim's algorithm adds them to the current tree (from green to yellow, and finally red) and the three images correspond to running Prim's algorithm for $n/3$, $2n/3$, and $n$ steps.}
    \label{fig:regular}
\end{figure}

\begin{figure}[htb]
    \centering
    \includegraphics[width=0.4\linewidth]{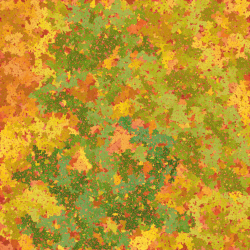}
    \caption{Simulation of Prim's algorithm on a $2000\times2000$ grid, starting from the root (in turquoise, in the middle). Each pixel corresponds to a single vertex and vertices and edges are coloured based on the step at which Prim's algorithm adds them to the current tree (from green to yellow, and finally red, with extra shaded fluctuations to highlight the local behaviour).}
    \label{fig:grid}
\end{figure}

\paragraph{\bf Addition and completion times}

We next use our results to describe the \textit{addition times} for edges in the neighbourhood of the root, and the \textit{completion time} of the neighbourhood of the root in the MST.
For $r\geq 0$, we define the {\em addition time}
\begin{align*}
    \tau_n(r, m)=\min\Big\{ k\colon \big|\prim_{k}(\rG_n)\cap B_r(\rG_n)\big|=m\Big\}
\end{align*}
to be the time that the $m^{\rm th}$ vertex in $ B_r(\rG_n)$ is added.
Additionally, for fixed $r\geq 0$, we define the {\em completion time} $C_n(r)$ by 
\begin{align*}
        C_n(r) &= \min\Big\{k: B_r\big(\prim_k(\rG_n)\big) = B_r\big(\prim_{n}(\rG_n)\big)\Big\}
\end{align*}
i.e., $C_n(r)$ denotes the first time that Prim's algorithm has discovered the complete $r$-neighbourhood of the MST. 
Our next theorem concerns the convergence in distribution of the addition and completion times:

\begin{theorem}[Limits of addition and completion times]\label{thm:addition-completion}
    Under the assumptions of Theorem \ref{thm:Prim convergence}, as $n\rightarrow \infty$, the process 
    \begin{align*}
        \left(\frac{\tau_n(r,m)}{n}, \frac{C_n(r)}{n}\right)_{m,r\geq 0}
    \end{align*}
    converges in distribution to the process of addition and completion times on the local limit $\rG$.
\end{theorem}
Note that $\tau_n(r,m)/n$ converges to zero {\em precisely} for the edges that are reached by the invasion percolation cluster on the local limit. 
Both limits can be made more explicit using only the structure of the limit.
In particular, one can check that
\begin{align*}
    \frac{C_n(r)}{n}\underset{n\rightarrow\infty}{\longrightarrow}\max\Big\{\theta\big(w(e)\big)\colon e\in E\big(B_r(\rMSF^{\rG})\big)\setminus E\big(\prim_\infty(\rG)\big)\Big\}\,,
\end{align*}
since any edge in $\prim_\infty(\rG)$ is explored in $o(n)$ steps and all edges connected to $\prim_\infty(\rG)$ via a path of weight less than $p$ are explored in the first $\theta(p)n+o(n)$ steps of Prim's algorithm.

\section{Background work and preliminary results}\label{sec:background}

In this section we provide  further details on the definitions from Section~\ref{sec:main result}.
We start with background notations and definitions regarding graphs in Section~\ref{sec:graphs}.
We then define the minimum spanning forest, the minimum spanning tree, the invasion percolation cluster, and the expanded invasion percolation cluster in Section~\ref{sec:msf + eipc}.
In Section~\ref{sec:comp speed}, we study properties related to the levels in an infinite rooted weighted graph, more precisely, the time it takes for Prim's algorithm to reach an infinite component at level $p$, and the regularity of the survival probability.
In Section~\ref{sec:lwc}, we introduce properties of the local convergence and discuss its consequences in Section~\ref{sec:local prop}.
Finally, in Section~\ref{sec:lpc} we extend the local convergence to the local process convergence from Theorem~\ref{thm:Prim convergence}.

\subsection{Graph notations}\label{sec:graphs}

A \textit{locally finite graph} (or simply \textit{graph}) $G=(V(G),E(G))$ is a pair of sets, where $V(G)$ is the arbitrary set of \textit{vertices}, $E(G)\subseteq\binom{V(G)}{2}=\{\{u,v\}\colon u,v\in V(G)\}$ is the set of \textit{edges}, composed of (unordered) pairs of vertices, and for any $v\in V(G)$, the set of neighbors $\{u\in V(G)\colon \{u,v\}\in E(G)\}$ is finite.

For any graph $G=(V(G),E(G))$, a pair $(G,o)$ with $o\in V(G)$ is called a \textit{rooted} graph, a pair $(G,w)$ with $w\colon E(G)\rightarrow\bR$ injective (or $w\colon \mathcal{E}\rightarrow\bR$ with $E(G)\subseteq\mathcal{E}$) is called a \textit{weighted} graph, and the subsequent triplet $(G,o,w)$ is called a \textit{rooted weighted} graph. We further consider $(G,o,w)$, $((G,o),w)$, and $((G,w),o)$ to be equivalent, allowing us to directly go from rooted or weighted graphs to rooted weighted graphs. 
In these settings, $o$ is referred to as the \textit{root} and $w$ as the \textit{weights}.
We often denote weighted graphs with a curly capital letter $\wG$ and rooted weighted graphs with a double bar capital letter $\rG$.
\smallskip

\invisible{\benoit{To be removed?}
\ch{For any graph $G$, we write $V(G)$ for its vertex set and $E(G)$ for its edge set and we extend these notations to rooted, weighted and rooted weighted graphs.}}
We let the size of $G$ be its number of vertices $|G|=|V(G)|$, and say that $G$ is finite when $|G|<\infty$.
For any two graphs $G$ and $G'$, we define their union and intersection by $G\otimes G'=(V(G)\otimes V(G'),E(G)\otimes E(G'))$ where $\otimes\in\{\cup,\cap\}$.
We then say that $G$ and $G'$ intersect when $V(G\cap G')\neq\varnothing$ (note that it might have zero edges).
Finally, we extend these notations by allowing one of the two graphs to be rooted, weighted, or rooted weighted, and by letting the union or the intersection be itself rooted, weighted, or rooted weighted with the same parameters.
Observe that the union requires the weights to be defined on both edge sets and that the intersection requires the root to be part of both vertex sets, which is always the case here.
\smallskip

Given a (possibly random) finite graph $G$, the \textit{standard extension} of $G$ is the random rooted weighted graph $\rG=(G,o,w)$ where $o$ is chosen uniformly at random among $V(G)$ and $(w(e))_{e\in E(G)}$ is a family of independent $\uniform$.
Similarly, the standard extension of a rooted (not necessarily finite) graph $(G,o)$ is the random rooted weighted graph $\rG=(G,o,w)$ where $(w(e))_{e\in E(G)}$ is a family of independent $\uniform$ random variables.
In these definitions, the weights in $w$ are always independent of $(G,o)$.
\smallskip

For any weighted graph $\wG=(G,w)$, we denote the graph obtained by only keeping edges of weight less than $p$ by $\wG(p)$, that is, $\wG(p)=(G_p,w)$ with $V(G_p)=V(G)$ and $E(G_p)=\{e\in E(G)\colon w(e)\leq p\}$;
we similarly define $\rG(p)=((G,w)(p),o)$ when $\rG=(G,o,w)$ is a rooted weighted graph.
Observe that, if $G$ is a graph and $\rG$ is its standard extension, then $G_p$ as previously defined has the distribution of $G$ percolated at level $p$, however we also implicitly couple all such $(G_p)_{p\in[0,1]}$ here so that $p\mapsto G_p$ is c\`adl\`ag.
\smallskip

For any weighted graph $\wG=(G,w)$ and any $v\in V(\wG)$, we denote the connected component of $v$ in $\wG(p)$ by $\comp^{\wG}_v(p)$.
Moreover, if $\wG$ is finite, we denote the connected components of $\wG(p)$ in decreasing order of size by $\comp^{\wG}_{(1)}(p),\comp^{\wG}_{(2)}(p),\ldots$  (breaking ties arbitrarily).
We also extend these notations to rooted weighted graphs.
In both cases (weighted or rooted weighted), the components are considered as standard graphs and we drop the weights (and the root) when considering these components.
Observe that, if $G$ is a graph and $\rG$ is its standard extension, then $\comp^{\rG}_v(0)=(\{v\},\varnothing)$ and $\comp^{\rG}_v(1)=\comp^{\rG}_{(1)}(1)=G$ for any $v\in V(G)$.

\subsection{Minimum spanning forest and expanded invasion percolation cluster}\label{sec:msf + eipc}

Let $\wG=(G,w)$ be a weighted graph.
We call the graph with vertex set $V(\MSF^{\wG})=V(G)$ and edge set defined by $e=\{u,v\}\in E(\MSF^{\wG})$ if and only if the two components $\comp^{\wG\setminus\{e\}}_u(w(e))$ and $\comp^{\wG\setminus\{e\}}_v(w(e))$ are disjoint and not both infinite (where $\wG\setminus\{e\}$ denotes the graph $\wG$ where the edge $e$ is removed) the \textit{minimum spanning forest} of $\wG$, and denote it $\MSF^{\wG}$.
For a rooted weighted graph $\rG=(G,o,w)$, we let $\MSF^{\rG}=\MSF^{(G,w)}$ be the minimum spanning forest of $\rG$ and further define $\rMSF^{\rG}=(\MSF^{\rG},o,w)$ to be the rooted weighted minimum spanning forest of $\rG$.
Since this definition may yield a disconnected forest when $G$ is infinite or disconnected, and we are only interested in its local limit, we actually reduce $\rMSF^{\rG}$ to the unique tree containing the root.
We still refer to it as a \textit{forest}, although it is actually a tree.
\smallskip

In the case where $\rG$ (or $\wG$) is finite and connected, the forest $\MSF^{\rG}$ is always connected and usually referred to as the \textit{minimum spanning tree} of $\rG$, which we denote $\MST^{\rG}$ here (or $\rMST^{\rG}$ for its rooted weighted counterpart).
There exist several efficient algorithms to compute the minimum spanning tree and we now present one of them, Prim's algorithm~\cite{Prim1957}, while not providing the proof that it indeed outputs the minimum spanning tree.
\smallskip

Given a connected rooted weighted graph $\rG=(G,o,w)$, we recursively define a sequence of trees as follows.
First, we start from $T_1=(\{o\},\varnothing)$.
Then, given $T_{k-1}=(V_{k-1},E_{k-1})$ with $k\leq|\rG|$, let $\{u_k,v_k\}$ be the edge of minimal weight among all edges with exactly one end in $T_{k-1}$.
Assume without loss of generality that $u_k\in V_{k-1}$ and $v_k\notin V_{k-1}$.
Then, we let $T_k=(V_{k-1}\cup\{v_k\},E_{k-1}\cup\{u_k,v_k\})$ be the tree obtained by adding this edge to $T_{k-1}$.
This defines a sequence of growing trees such that $T_{|\rG|}$ is the minimum spanning tree of $\rG$ when the graph is finite.
If that is the case, we also let $T_k=T_{|\rG|}$ when $k\geq|\rG|$.
For any $k\geq1$, we now let $\prim_k(\rG)$ be the rooted weighted outcome of this process after $k$ steps, that is
\begin{align*}
    \prim_k(\rG)=(T_k,o,w)\,.
\end{align*}
We further extend this definition to $k=\infty$ by letting
\begin{align*}
    T_\infty=\left(\bigcup_{k\geq0}V_k,\bigcup_{k\geq0}E_k\right)
\end{align*}
and setting $\prim_\infty(\rG)=(T_\infty,o,w)$ as before.
We observe that $\prim_\infty(\rG)$ is always a subtree of $\rMSF^{\rG}$.

When $\rG$ is infinite, $\prim_\infty(\rG)$ is actually a commonly studied object, usually referred to as the \textit{invasion percolation cluster}.
For this reason, we now define the \textit{expanded invasion percolation cluster}, written as $\rMSF^{\rG}_+(p)$, to be the union of the invasion percolation cluster and the edges in the minimal spanning forest of weight at most $p$, denoted by $E(\rMSF^{\rG}(p))$, obtained by only keeping edges in $\rMSF^{\rG}$ of weight less than $p$.
More precisely, $\rMSF^{\rG}_+(p)$ is the rooted weighted graph which has the same vertex set, root, and weights as $\rG$, and with edge set defined by
\begin{align*}
    E\big(\rMSF^{\rG}_+(p)\big)=E\big(\prim_\infty(\rG)\big)\cup E\big(\rMSF^{\rG}(p)\big)\,.
\end{align*}
We again observe that the expanded invasion percolation cluster is likely not connected, and thus only consider the component containing the root, making $\rMSF^{\rG}_+(p)$ a tree again.

\subsection{Components in an infinite graph}\label{sec:comp speed}

In this section, we consider an infinite rooted weighted graph $\rG=(G,o,w)$.
Following the definition of $\theta=\theta^{\rG}$ in~\eqref{eq:theta}, we are interested in the time it takes for Prim's algorithm to reach an infinite component from the root.
More precisely, for any $p\in[0,1]$, we let $\speed(p)=\speed^{\rG}(p)$ be the random variable
\begin{align}\label{eq:speed}
    \speed(p)&=\min\Big\{k:\exists v\in V\big(\prim_k(\rG)\big),\big|\comp_v(p)\big|=\infty\Big\}-1\,.
\end{align}
From this definition, we observe that $\speed(p)=\infty$ when Prim's algorithm {\em never} explores an infinite component at percolation level $p$, and $\speed(p)=0$ is equivalent to the root $o$ already belonging to an infinite component at percolation level $p$.
In particular, it follows that $p\mapsto\speed(p)$ is decreasing (under the coupling of $\rG(p)$ using its edge weights) with $\speed(0)=\infty$ and $\speed(1)$ equals $0$ if $\rG$ is infinite, $\infty$ otherwise.
Finally, when $\speed(p)<\infty$, we let $v(p)=v^{\rG}(p)$ be the $(\speed(p)+1)$-th vertex explored during Prim's algorithm and we note that this implies that $|\comp^{\rG}_{v(p)}(p)|=\infty$;
in the case $\speed(p)=\infty$, for convenience, we let $v(p)=\infty$ be a ``boundary'' vertex and define $\comp^{\rG}_{v(p)}(p)=(\varnothing,\varnothing)$ to be the empty graph.

In order to see the importance of $\speed(p)$, we start with a result showing the relation between its value and Assumption~\ref{ass:speed}:

\begin{proposition}[Infinite components reached by Prim]\label{prop:speed}
    Let $\rG$ be a rooted weighted graph.
    Then, for any $p\in[0,1]$,
    \begin{align*}
        \speed(p)=\infty~\Longleftrightarrow~\forall v\in\rMSF^{\rG}:\big|\comp^{\rG}_v(p)\big|<\infty\,.
    \end{align*}
    In words, Prim's algorithm never reaches an infinite component at level $p$ if and only if all components at level $p$ reachable from $o$ in the minimum spanning forest have finite sizes.
\end{proposition}
Assumption \ref{ass:speed} states that the event on the right-hand side has probability zero when $\theta(p)>0$. Thus, in this case, we obtain that $\speed(p)<\infty$ a.s.

\begin{proof}
    Note first that $\comp^{\rG}_v(p)\cap\MSF^{\rG}=\MSF^{\rG(p)}$ so that $|\comp^{\rG}_v(p)|=|\comp^{\rMSF^{\rG}}_v(p)|$ (if they are both infinite, they are not necessarily equal as $\rMSF^{\rG}$ is only the connected component of the root).
    We now observe that the vertices and edges found by Prim's algorithm starting from $o$ in $\rG$ are all in $\rMSF^{\rG}$.
    \smallskip
    
    From this observation, the left implication directly follows, since there is no infinite component for Prim's algorithm to encounter.
    For the converse implication, assume that $\rMSF^{\rG}$ has infinite components at level $p$ and let $\rMSF^\infty$ be their union.
    We now show that Prim's algorithm eventually reaches $\rMSF^\infty$.

    If $o\in\rMSF^\infty$, then it is trivially true, so we now assume that $o\notin\rMSF^\infty$.
    Let $o=v_0,\ldots,v_k$ be a shortest path from $o$ to $\rMSF^\infty$ and let $e=\{v_i,v_{i+1}\}$ be the edge of maximal weight in this path.
    Since we know that $o\notin\rMSF^\infty$ and $e$ is the edge of maximal weight, it follows that $w(e)>p$.
    But then, in that case $\comp^{\rG}_{v_{i+1}}(w(e))$ contains the infinite component at level $p$ included in $\rMSF^\infty$, so it is infinite.
    By the definition of the minimum spanning forest from Section~\ref{sec:graphs}, this implies that $|\comp^{\rG}_{v_i}(w(e))|=|\comp^{\rG}_o(w(e))|<\infty$, and so Prim's algorithm eventually visits the edge $e$.
    Using this argument inductively, we see that Prim's algorithm eventually reaches $\rMSF^\infty$, as desired.
\end{proof}

We now use $\speed(p)$ to characterise the expanded invasion percolation cluster, and highlight its relation to the invasion percolation cluster (IPC) and Prim's algorithm:


\begin{lemma}[Decomposition of expanded IPC]\label{lem:eipc struct}
    Let $\rG$ be a rooted weighted graph, $\rMSF^{\rG}$ its rooted weighted minimum spanning forest, and $\rMSF^{\rG}_+(p)$ its expanded invasion percolation cluster as defined in Section~\ref{sec:msf + eipc}.
    Further define $\speed(p)=\speed^{\rG}(p)$ as in~\eqref{eq:speed} and $v(p)=v^{\rG}(p)$ as the $(\speed(p)+1)$-th vertex explored in Prim's algorithm.
    Then, for any $p\in[0,1]$,
    \begin{align*}
        \rMSF^{\rG}_+(p)&=\prim_{\speed(p)+1}(\rG)\cup\left(\MSF^{\rG}\cap\comp^{\rG}_{v(p)}(p)\right)\,,
    \end{align*}
    where, when $\speed(p)=\infty$, this formula equals
    \begin{align*}
        \rMSF^{\rG}_+(p)&=\prim_\infty(\rG).
    \end{align*}
\end{lemma}

\begin{proof}
    Consider first the case $\speed(p)=\infty$.
    We already know from its definition in Section~\ref{sec:msf + eipc} that $\prim_\infty(\rG)$ is a subtree of $\rMSF^{\rG}_+(p)$.
    For the other inclusion, let $e=\{u,v\}$ be an edge on the boundary of $\prim_\infty(\rG)$ in $\rMSF^{\rG}$, so that $u$ is explored by Prim's algorithm, while $v$ is not.
    Using the definition of Prim's algorithm, we observe that $\comp^{\rG\setminus\{e\}}_u(w(e))$ contains an infinite component of $\prim_\infty(\rG)$ and so is infinite.
    Using the result from Proposition~\ref{prop:speed}, we see that this implies that $w(e)>p$ and so $e\notin E(\rMSF_+(p))$ as desired.
    
    Next, consider the case where $\speed(p)<\infty$.
    The left inclusion once again directly follows from the definition of $\rMSF^{\rG}_+(p)$, so we focus on the right inclusion.
    First of all, since Prim's algorithm reaches the component $\comp^{\rG}_{v(p)}(p)$ after $\speed(p)+1$ steps and this component is infinite, it remains within it for the rest of the time.
    Using that Prim's algorithm is a subtree of the minimum spanning forest, we see that
    \begin{align*}
        \prim_{\infty}(\rG)\subseteq\prim_{\speed(p)+1}(\rG)\cup\left(\MSF^{\rG}\cup\comp^{\rG}_{v(p)}(p)\right)\,.
    \end{align*}
    Write $T$ for the tree on the right-hand side of the previous equation (i.e., the target object from the lemma) and consider an edge $e=\{u,v\}$ from $\rMSF^{\rG}$ such that $u$ is in $T$ and $v$ is not.
    Then either $u\in\comp^{\rG}_{v(p)}(p)$ and since $v$ does not belong to $T$, it follows that $w(e)>p$ so $e$ does not belong to $\rMSF^{\rG}_+(p)$.
    Otherwise, $u\notin\comp^{\rG}_{v(p)}(p)$ and so $\speed(p)\geq1$.
    In that case, there is an edge between $o$ and $\comp^{\rG}_{v(p)}(p)$ in $\rMSF^{\rG}$ whose weight is larger than $p$.
    But then, the edge $e$ has to be heavier than $p$, otherwise it would have been incorporated in Prim's algorithm.
    In both cases, we observe that $w(e)>p$ and so $e$ does not belong to $\rMSF^{\rG}_+(p)$, proving the second inclusion.
\end{proof}

\subsection{Local limits}\label{sec:lwc}

For any $r\geq0$ and any rooted graph $(G,o)$, the ball of radius $r$ of $(G,o)$ denoted $B_r(G,o)$ is the rooted subgraph of $G$ composed of all vertices that can be reach from $o$ using only $r$ edges;
for example, $B_0(G,o)=((\{o\},\varnothing),o)$ and $B_1(G,o)$ is defined by $V(B_1(G,o))=\{o\}\cup N$, where $N$ is the set of neighbours of $o$ in $G$, and
\begin{align*}
    E\big(B_1(G,o)\big)=\Big\{\{o,v\}:v\in N\Big\}\cup\left(E(G)\cap\binom{N}{2}\right)\,.
\end{align*}
For a rooted weighted graph $\rG=(G,o,w)$, we extend the notations and denote $B_r(\rG)=(B_r(G,o),w)$ for the rooted weighted subgraph of $\rG$ composed of vertices at distance $r$ from $o$.
\smallskip

Two graphs $G$ and $G'$ are said to be \textit{isomorphic} if there exists a bijection $\varphi\colon V(G)\rightarrow V(G')$ such that $e=\{u,v\}\in E(G)$ if and only if $\varphi(e)=\{\varphi(u),\varphi(v)\}\in E(G')$.
Two rooted graphs $(G,o)$ and $(G',o')$ are said to be isomorphic if $G$ and $G'$ are isomorphic and the previous bijection $\varphi$ can be chosen to satisfy $\varphi(o)=o'$.
Finally, while this definition could directly be extended to rooted weighted graphs, since the weights considered here are continuous, we prefer using an approximation, further explained below.
\smallskip

Consider two rooted weighted graphs $\rG=(G,o,w)$ and $\rG'=(G',o',w')$.
Then, $\varphi$ is said to be a $\varepsilon$-isomorphism between $\rG$ and $\rG'$ if $\varphi$ satisfies the previous isomorphic condition between $(G,o)$ and $(G',o')$ and
\begin{align*}
    \big|w(e)-w'(\varphi(e))\big|\leq\varepsilon
\end{align*}
for any $e\in E(G)$.
We then write $\rG\equiv_r\rG'$ if there exists a $1/r$-isomorphism between $B_r(\rG)$ and $B_r(\rG')$.
If that is the case, this means that the unweighted versions of $\rG$ and $\rG'$ are isomorphic on a ball of radius $r$ and that such isomorphism can be chosen so that edges are mapped to edges with weights that are at most $1/r$ apart.
We also naturally extend this notation to unweighted graphs by saying that $(G,o)\equiv_r(G',o')$ if and only if $(G,o,\mathbf{1})\equiv_r(G',o',\mathbf{1})$, where $\mathbf{1}$ is the constant function equal to $1$.
\smallskip

For two rooted weighted graphs $\rG$ and $\rG'$, define their distance to be $1/(r+1)$, where $r$ is the maximal radius on which they are equivalent, that is
\begin{align}\label{eq:def local metric}
    d(\rG,\rG')&=\left(1+\sup\Big\{r\geq0:\rG\equiv_r\rG'\Big\}\right)^{-1}\,.
\end{align}
Then, the metric space of rooted weighted graphs (up to isomorphisms), endowed with the distance $d$, is a Polish space.
Indeed, both the space of rooted graphs endowed with $d$ (by giving constant weights equal to $1$) and $[0,1]$ endowed with the distance in absolute value are Polish, and so it is not difficult to verify that the space of rooted weighted graphs is Polish.
Thus, we can consider convergence, random variables, and random processes on it.
In particular, this allows us to define the standard local weak topology, as introduced in~\cite{Benjamini2001}, as follows.
\smallskip

A sequence of (random) weighted graphs $(\wG_n)_{n\geq1}$ converges \textit{locally weakly} to the rooted weighted graph $\rG$ if, for any bounded continuous function $h$ with respect to $d$,
\begin{align*}
    \expec\big[h(\wG_n,o_n)\big]\longrightarrow\expec[h(\rG)]\,,
\end{align*}
where $o_n$ is a vertex from $V(\wG_n)$ chosen uniformly at random.
When that is the case, we denote it by $\wG_n\lwc\rG$.
We further extend this notation to allow the convergence to occur not only in expectation, but also in probability, i.e., for random graphs $\wG_n$, 
\begin{equation*}
    \expec\big[h(\wG_n,o_n)\mid \wG_n\big]\stackrel{\scriptscriptstyle\prob}{\longrightarrow}\expec[h(\rG)],
\end{equation*}
where now the expectation is with respect to the random root only (i.e., now with respect to the randomness of the graph).
When that is the case, we say that $(\wG_n)_{n\geq1}$ converges \textit{locally in probability} towards $\rG$ and denote it by $\wG_n\lcp\rG$.
Similarly, we extend these two notations to rooted weighted graphs (that is, $\rG_n\lwc\rG$ and $\rG_n\lcp\rG$), where it is always the case that the root is chosen uniformly at random among the vertex set of $G_n$. We also later apply a function on $\rG_n$ (namely, Prim's algorithm) but the root is always chosen in advance, and uniformly in $V(G_n)$.
Finally, when considering unweighted graphs, we simply add constant edge weights equal to $1$, and use the same notations and terminology.
Interestingly, the two notions of weighted and unweighted local weak convergence are strongly related, as specified by the following lemma:

\begin{lemma}[Local convergence of random graphs with i.i.d.\ weights]\label{lem:lwc equiv}
    Let $(G_n)_{n\geq1}$ be a sequence of graphs and $(G,o)$ a rooted graph.
    Further endow these graphs with independent and identically distributed edge weights to obtain $(\wG_n)_{n\geq1}=(G_n,w_n)_{n\geq1}$ and $\rG=(G,o,w)$, where the weights have the same distributions.
    Then weighted and unweighted local weak convergence are equivalent, that is
    \begin{align*}
        G_n\lwc(G,o)~\Longleftrightarrow~\wG_n\lwc\rG\,.
    \end{align*}
\end{lemma}

\begin{proof} We give a direct proof, but this can also be proved using a coupling argument instead.
    Fix $r\geq0$ and a rooted weighted graph $\rH=(H,\rho,\xi)$.
    Further write $W$ for a random variable distributed according to the common weight distribution.
    Then, given the random graph $\rG=(G,o,w)$ and using the definition of $\equiv_r$, the event $\{\rG\equiv_r\rH\}$ can be split according to the existence of the isomorphism and the weights on this isomorphism being close to each other:
    \begin{align*}
        \prob\Big(\rG\equiv_r\rH\Big)&=\prob\Big(\rG\equiv_r\rH~\Big|~(G,o)\equiv_r(H,\rho)\Big)\prob\Big((G,o)\equiv_r(H,\rho)\Big)\\
        &=\left(\prod_{e\in E(B_r(\rH))}\prob\left(\big|W-\xi(e)\big|\leq\frac{1}{r}\right)\right)\prob\Big((G,o)\equiv_r(H,\rho)\Big)\,.
    \end{align*}
    Denote the product of probabilities in the last equation by $P(\rH)$.
    The same argument as before can then be applied to $(\wG_n,v)$ for any $n\geq1$ and $v\in V(\wG_n)$, leading to
    \begin{align*}
        &\frac{1}{|\wG_n|}\sum_{v\in V(\wG_n)}\prob\Big((\wG_n,v)\equiv_r\rH\Big)-\prob\Big(\rG\equiv_r\rH\Big)\\
        &=P(\rH)\cdot\left(\frac{1}{|G_n|}\sum_{v\in V(G_n)}\prob\Big((G_n,v)\equiv_r(H,\rho)\Big)-\prob\Big((G,o)\equiv_r(H,\rho)\Big)\right)\,.
    \end{align*}
    To conclude this proof, use~\cite[Theorem 2.16]{VanderHofstad2024} to restrict our study to $\rH$ such that $P(\rH)>0$, thus proving that the left-hand side converges to $0$ if and only if the right-hand side does as well.
    It now suffices to notice that the left-, respectively, right-hand sides converge to $0$ for all such $\rH$ if and only if the weighted, respectively unweighted, graph sequence converges, which is exactly the desired statement.
\end{proof}

\subsection{Properties of the local limit}\label{sec:local prop}

In this section we explore some of the known properties regarding the local limits of graph sequences.
As a first step, and because it is an essential part of our work, we provide the following theorem stating that the local limit of the minimum spanning tree is the minimum spanning forest:

\begin{theorem}[{\cite[Theorem~5.4]{Aldous2004}}]\label{thm:mst and msf}
    Let $(\wG_n)_{n\geq1}$ be a sequence of weighted graphs which converges locally weakly towards $\rG$.
    Then, under the standard product topology,
    \begin{align*}
        \big(\wG_n,\MST^{\wG_n}\big)\lwc\big(\rG,\MSF^{\rG}\big)\,.
    \end{align*}
\end{theorem}

It is easy to see that, when $\rG_n$ is the standard extension of $G_n$, the previous result implies that the paired rooted weighted convergence also occurs, i.e.,
\begin{align*}
    \big(\rG_n,\rMST^{\rG_n}\big)\lwc\big(\rG,\rMSF^{\rG}\big)\,.
\end{align*}
Furthermore, using a similar argument as that of the original proof, the previous convergence of measures can be extended, thanks to Skorohod's representation theorem, to an almost-sure convergence on an appropriate probability space.
Under this coupling, for any $r\geq1$, there exists $n_0$ such that, for $n\geq n_0$, we have $\rG_n\equiv_r\rG$ and $\rMST^{\rG_n}\equiv_r\rMSF^{\rG}$.
Naturally, this can be extended to $r=r_n$ growing slowly enough.
We say that a sequence of functions $k=(k_n(\cdot))_{n\geq1}$ is a \textit{linearly growing sequence} with respect to $(\rG_n)_{n\geq1}$ (or equivalently with respect to $(G_n)_{n\geq1}$) if, for any $t\in[0,1]$, we have $\lim_{n\rightarrow\infty}k_n(t)/n=t$, and further $(k_n(0))_{n\geq1}$ is a slowly growing radius diverging to infinity such that there exists a coupling satisfying
\begin{align}\label{eq:slowly growing}
    \prob\Big(\rG_m\equiv_{k_n(0)}\rG\textrm{ and }\rMSF^{\rG_m}\equiv_{k_n(0)}\rMSF^{\rG}\Big)=1
\end{align}
for all $m\geq n$.
In other words, this definition means that $k_n(t)=nt+o(n)$ where the $o(n)$ term when $t=0$ is diverging at a rate allowing us to couple $(\rG_n,\rMST^{\rG_n})$ with $(\rG,\rMSF^{\rG})$.
\smallskip

After having stated the fundamental result regarding the convergence of the minimum spanning tree, and having defined linearly growing sequences, we now complete this section with applications of Assumption~\ref{ass:local}.
The first result states that the largest component of the finite sequence of graphs asymptotically covers the same proportion of vertices as the probability for the graph to be infinite in the limit:

\begin{proposition}[{\cite[Theorem 2.28]{VanderHofstad2024}}]\label{prop:theta}
    Let $(G_n)_{n\geq1}$ be locally converging in probability to $(G,o)$.
    Let $(\rG_n)_{n\geq1}$ and $\rG$ be the respective standard expansions of $(G_n)_{n\geq1}$ and $(G,o)$, and define $\theta$ as in~\eqref{eq:theta}.
    Then, if $(G_n)_{n\geq1}$ satisfies Assumption~\ref{ass:local}, for any $p\in[0,1]$,
    \begin{align*}
        \frac{\big|\comp^{\rG_n}_{(1)}(p)\big|}{n}\probc\theta(p)
        \hspace{0.5cm}\textrm{and}\hspace{0.5cm}
        \frac{\big|\comp^{\rG_n}_{(2)}(p)\big|}{n}\probc0\,.
    \end{align*}
    In fact, if Assumption~\ref{ass:local} holds for some $p\in[0,1]$, then also Proposition \ref{prop:theta} holds for that $p$.
\end{proposition}

This provides a first characterisation of the sizes of components in the finite graphs.
It is interesting to note that the fact that the second largest component has a size of a smaller order than $n$ also means that this is true for every subsequent component.
We now provide a more local characteristic of the largest and other components:

\begin{proposition}[{\cite[Theorem 2.32]{VanderHofstad2024}}]\label{prop:local giant}
    Let $(G_n)_{n\geq1}$ be locally converging in probability to $(G,o)$.
    Let $(\rG_n)_{n\geq1}$ and $\rG$ be the respective standard expansions of $(G_n)_{n\geq1}$ and $(G,o)$, and define $\theta$ as in~\eqref{eq:theta}.
    Then, if $(G_n)_{n\geq1}$ satisfies Assumption~\ref{ass:local}, for any $p\in[0,1]$,
    \begin{align*}
        \comp^{\rG_n}_{(1)}(p)\lcp\comp^{\rG}_\root(p)~\Big|~\big|\comp^{\rG}_\root(p)\big|=\infty
        \hspace{0.5cm}\textrm{and}\hspace{0.5cm}
        \bigcup_{k\geq2}\comp^{\rG_n}_{(k)}(p)\lcp\comp^{\rG}_\root(p)~\Big|~\big|\comp^{\rG}_\root(p)\big|<\infty\,.
    \end{align*}
\end{proposition}

This second proposition is very useful as it states that we can approximately know whether the root belongs to the largest component or not simply from observing the neighbourhood of a node.
This is a key point in the proof of Proposition~\ref{prop:no large} below, stating that the only large component in a neighbourhood is the largest component:

\begin{proposition}\label{prop:no large}
    Let $(G_n)_{n\geq1}$ be locally converging in probability to $(G,o)$.
    Let $(\rG_n)_{n\geq1}$ and $\rG$ be the respective standard expansions of $(G_n)_{n\geq1}$ and $(G,o)$.
    Then, if $(G_n)_{n\geq1}$ satisfies Assumption~\ref{ass:local}, for any $p\in[0,1]$ and $r\geq0$,
    \begin{align*}
        \lim_{k\rightarrow\infty}\limsup_{n\rightarrow\infty}\prob\Big(\exists v\in B_r(\rG_n)\setminus\comp^{\rG_n}_{(1)}(p):|\comp^{\rG_n}_v(p)|\geq k\Big)=0\,.
    \end{align*}
    In other words, for any $\varepsilon>0$, there exists $k_0$ and $n_0=n_0(k_0)$ such that, for all $k\geq k_0$ and $n\geq n_0$,
    \begin{align*}
        \prob\Big(\exists v\in B_r(\rG_n)\setminus\comp^{\rG_n}_{(1)}(p):|\comp^{\rG_n}_v(p)|\geq k\Big)\leq\varepsilon\,.
    \end{align*}
\end{proposition}

\begin{proof}
    Fix $\varepsilon>0$.
    First observe that, for $M>0$ large enough and for all $n\geq1$,
    \begin{align*}
        \prob\Big(\big|B_{2r}(\rG_n)\big|>M\Big)\leq\varepsilon\,.
    \end{align*}
    It follows that
    \begin{align*}
        &\prob\Big(\exists v\in B_r(\rG_n)\setminus\comp^{\rG_n}_{(1)}(p):|\comp^{\rG_n}_v(p)|\geq k\Big)\\
        &\pskip\leq\varepsilon+\prob\Big(\exists v\in B_r(\rG_n)\setminus\comp^{\rG_n}_{(1)}(p):|\comp^{\rG_n}_v(p)|\geq k,\big|B_{2r}(\rG_n)\big|\leq M\Big)\\
        &\pskip\leq\varepsilon+\prob\Big(\exists v\in B_r(\rG_n):v\in\bigcup_{i\geq2}\comp^{\rG_n}_{(i)}(p),|\comp^{\rG_n}_v(p)|\geq k,\big|B_r(G_n,v)\big|\leq M\Big)\,,
    \end{align*}
    where we have used that if the ball of radius $2r$ has size at most $M$, then so does any ball of radius $r$ of a vertex at distance at most $r$ from the original root.
    Furthermore, using that $v\in B_r(\rG_n)=B_r(G_n,o_n,w_n)$ if and only if $o_n\in B_r(G_n,v)$, and by summing over all possible $v\in V(\rG_n)$, the previous bound becomes
    \begin{align*}
        &\prob\Big(\exists v\in B_r(\rG_n)\setminus\comp^{\rG_n}_{(1)}(p):|\comp^{\rG_n}_v(p)|\geq k\Big)\\
        &\pskip\leq\varepsilon+\sum_{v\in V(G_n)}\prob\Big(o_n\in B_r(G_n,v),v\in\bigcup_{i\geq2}\comp^{\rG_n}_{(i)}(p),|\comp^{\rG_n}_v(p)|\geq k,\big|B_r(G_n,v)\big|\leq M\Big)\,.
    \end{align*}
    Observe that only the first event of the probability depends explicitly on $o_n$. Thus, using that $o_n$ is uniform over the set of vertices, it follows that
    \begin{align*}
        &\prob\Big(\exists v\in B_r(\rG_n)\setminus\comp^{\rG_n}_{(1)}(p):|\comp^{\rG_n}_v(p)|\geq k\Big)\\
        &\pskip\leq\varepsilon+\sum_{v\in V(G_n)}\expec\left[\frac{|B_r(G_n,v)|}{|V(G_n)|}\cdot\rI_{\left\{v\in\bigcup_{i\geq2}\comp^{\rG_n}_{(i)}(p),|\comp^{\rG_n}_v(p)|\geq k,|B_r(G_n,v)|\leq M\right\}}\right]\\
        &\pskip\leq\varepsilon+\frac{M}{|V(G_n)|}\sum_{v\in V(G_n)}\prob\Big(v\in\bigcup_{i\geq2}\comp^{\rG_n}_{(i)}(p),|\comp^{\rG_n}_v(p)|\geq k\Big)\,.
    \end{align*}
    To conclude the proof, we simply note that Proposition~\ref{prop:local giant} tells us that
    \begin{align*}
        \lim_{n\rightarrow\infty}\frac{1}{|V(G_n)|}\sum_{v\in V(G_n)}\prob\Big(v\in\bigcup_{i\geq2}\comp^{\rG_n}_{(i)}(p),|\comp^{\rG_n}_v(p)|\geq k\Big)=\prob\Big(|\comp^{\rG}_o(p)|\geq k~\Big|~|\comp^{\rG}_o(p)|<\infty\Big),
    \end{align*}
    and the right-hand side converges to $0$ as $k\rightarrow\infty$.
    This proves the desired convergence.
\end{proof}

\subsection{Process convergence}
\label{sec:lpc}

We conclude Section \ref{sec:background} with a short background on process convergence for sequences of stochastic processes as used in Theorem~\ref{thm:Prim convergence}.
The main strategy is to use the standard \textit{Skorohod topology} (also called \textit{Skorohod $J_1$ topology}) to allow for processes on Polish spaces (like the space of rooted weighted graphs) to be defined.
We hereafter simplify the process convergence to three conditions and refer to~\cite[Chapter~3]{Billingsley1999} for more details on this topology.

Let $(\rG_n(t))_{n\geq1,t\in[0,1]}$ be a sequence of cadlag processes on the space of rooted weighted graphs and recall the definition of $d$ from~\eqref{eq:def local metric}.
Then we say that it converges towards the cadlag process $\rG=(\rG(t))_{t\in[0,1]}$, and write it $(\rG_n(t))_{t\in[0,1]}\lpc(\rG(t))_{t\in[0,1]}$ if and only if it satisfies the following three conditions.
\begin{itemize}
    \item[$\rhd$] The multi-dimensional process converges: for all integer $r\geq1$, sets of time $t_1,\ldots,t_k$, and rooted weighted graphs $H_1,\ldots,H_k$, we have
    \begin{align}\label{eq:lpc multi-dim}
        \lim_{n\rightarrow\infty}\prob\Big(\rG_n(t_1)\equiv_rH_1,\ldots,\rG_n(t_k)\equiv_rH_k\Big)=\prob\Big(\rG(t_1)\equiv_rH_1,\ldots,\rG(t_k)\equiv_rH_k\Big)\,.
    \end{align}
    \item[$\rhd$] The limiting process is continuous at $1$: for all $\varepsilon>0$, we have
    \begin{align}\label{eq:lpc one-cont}
        \lim_{\delta\rightarrow0}\prob\Big(d\big(\rG(1),\rG(1-\delta)\big)\geq\varepsilon\Big)=0
    \end{align}
    \item[$\rhd$] The sequence does not have high increments: for all $\varepsilon>0$, there exists $\delta_n\in(0,1)$ such that
    \begin{align}\label{eq:lpc small jump}
        \lim_{n\rightarrow\infty}\prob\Big(\exists t<t'<t''<t+\delta_n:d\big(G_n(t),G_n(t')\big)\geq\varepsilon\textrm{ and }d\big(G_n(t'),G_n(t'')\big)\geq\varepsilon\Big)=0\,.
    \end{align}
\end{itemize}
This definition for the local process convergence arises from~\cite[Theorem~13.3]{Billingsley1999} and the latter two conditions~\eqref{eq:lpc one-cont} and~\eqref{eq:lpc small jump} can be understood as the graph process being tight.
We refer to the extensive discussion of dynamic local convergence in \cite{Milewska2023}.
\smallskip

It is worth noting that while this notion of convergence can be seen as very generic, the structure of $\rG_n(t)$ is specific in our case.
Indeed, it is always of the form $f_t(G_n,o_n,w_n)$ where $(G_n)_{n\geq1}$ is a generic sequence of graphs, $o_n$ is a uniformly chosen vertex, $w_n$ are independent $\uniform$ edge weights, and $f_{t}$ is a function mapping rooted weighted graphs unto rooted trees (more precisely, Prim's algorithm applied a partial number of steps).

\section{Convergence of Prim's algorithm}\label{sec:thm}

The goal of this section is to prove Theorem~\ref{thm:Prim convergence}. Let us start by recalling some notations that we use throughout this section.
\medskip

For the rest of the section, we assume that $(G_n)_{n\geq1}$ converges locally in probability towards $(G,o)$ and we let $(\rG_n)_{n\geq1}$ and $\rG$ be the respective standard expansions of $(G_n)_{n\geq1}$ and $(G,o)$.
Using Skorohod's representation theorem, we assume without loss of generality that $(\rG_n,\rMST^{\rG_n})$ converges almost-surely towards $(\rG,\rMSF^{\rG})$ and we let $k=(k_n(\cdot))_{n\geq1}$ be a linearly growing sequence satisfying~\eqref{eq:slowly growing}.
We let $\theta=\theta^{\rG}$, $\theta^{-1}$, and $\speed=\speed^{\rG}$ be defined as in~\eqref{eq:theta}, \eqref{eq:theta inv}, and~\eqref{eq:speed} respectively. Finally, assume that $\rG$ satisfies Assumption~\ref{ass:speed} and $(\rG_n)_{n\geq1}$ satisfies Assumption~\ref{ass:local} (Assumption~\ref{ass:smooth} is not necessary for the first results).
\smallskip

The rest of the section is organised as follows.
In Section~\ref{sec:special Prim}, we cover a special convergence result for Prim's algorithm which we then use in Section~\ref{sec:marginals} to first prove the necessary one-dimensional convergence for Theorem~\ref{thm:Prim convergence}. In turn, we then extend one-dimensional to multi-dimensional convergence as in~\eqref{eq:lpc multi-dim}.
Finally, we conclude with the proof that Prim's algorithm and its limit satisfy the two conditions~\eqref{eq:lpc one-cont} and~\eqref{eq:lpc small jump}, thus proving convergence as a process.

\subsection{Prim's algorithm with an appropriate number of steps}\label{sec:special Prim}

In this section, we provide a special case for the convergence of Prim's algorithm towards the expanded invasion percolation cluster, which is the first key step in the proof of Theorem~\ref{thm:Prim convergence}. We state this result as follows:

\begin{proposition}\label{prop:main result}
    Let Assumptions~\ref{ass:speed} and \ref{ass:local} hold.
    Using the notations from the beginning of Section~\ref{sec:thm}, for any $r\geq1$ and $p$ such that $\theta(p)>0$,
    \begin{align*}
        \lim_{n\rightarrow\infty}\prob\Big(\prim_{\speed(p)+\left|\comp^{\rG_n}_{(1)}(p)\right|}(\rG_n)\equiv_r\rMSF^{\rG}_+(p)\Big)=1\,.
    \end{align*}
\end{proposition}

\begin{proof}
    \def\epsdiv{3}
    Fix $r\geq1$, $\varepsilon>0$ and $p$ satisfying $\theta(p)>0$.
    Recall that $k_n=k_n(0)$ diverges and that $\rG_n\equiv_{k_n}\rG$ and $\rMST^{\rG_n}\equiv_{k_n}\rMSF^{\rG}$ almost-surely.
    Using that $\theta(p)>0$ along with Assumption~\ref{ass:speed} and Proposition~\ref{prop:speed}, $\speed(p)$ is finite so there exists $n_0$ such that, for all $n\geq n_0$,
    \begin{align*}
        \prob\big(\speed(p)>k_n\big)\leq\frac{\varepsilon}{\epsdiv}\,.
    \end{align*}
    Moreover, by applying Proposition~\ref{prop:no large} with $r=k_{n_0}$, we know that there exists $n_1$ and $k_0$ such that, for all $n\geq n_1$ and $k\geq k_0$,
    \begin{align*}
        \prob\Big(\exists v\in B_{k_{n_0}}(\rG_n)\setminus\comp^{\rG_n}_{(1)}(p):|\comp^{\rG_n}_v(p)|\geq k\Big)\leq\frac{\varepsilon}{\epsdiv}\,.
    \end{align*}
    Further, set $n_2$ such that, for all $n\geq n_2$, we have $k_n\geq k_{n_0}+k_0$.
    Finally, using that $\rG$ is locally finite, define $n_3$ such that, for all $n\geq n_3$,
    \begin{align*}
        \prob\left(\inf\Big\{|w_1-w_2|:w_1\neq w_2\in\{p\}\cup\big\{w(e):e\in E\big(B_{k_{n_2}}(\rG)\big)\big\}\Big\}\leq\frac{1}{2k_n}\right)\leq\frac{\varepsilon}{\epsdiv}\,.
    \end{align*}
    In words, the distance between all edge weights in $B_{k_{n_2}}(\rG)$ as well as the distance between these edge weights and $p$ is at least $1/2k_{n_3}$.
    \smallskip
    
    For all $n$, we write $E_n$ for the union of all the bad events considered previously, that is $\speed(p)>k_{n_0}$, $B_{k_{n_0}}(\rG_n)$ has a component of size at least $k_0$ which is not its largest component, and the weights of the edges of $B_{k_{n_2}}(\rG)$ are at a distance less than $1/(2k_{n_3})$ from each other and from $p$.
    The previous definitions of $n_0,n_1,n_2,n_3$ imply that, for any $n\geq\max\{n_0,n_1,n_2,n_3\}$, $\prob(E_n)\leq\varepsilon$.
    We now show that, on the event $E_n^c$, the event stated in the proposition holds.
    \smallskip

    We recall that $v(p)=v^{\rG}(p)$ was defined in Section~\ref{sec:comp speed} as the first vertex in an infinite component explored by Prim's algorithm on $\rG$.
    In particular, on the event $E_n^c$, we have that $\speed(p)\leq k_{n_0}<\infty$ and so $|\comp^{\rG}_{v(p)}(p)|=\infty$.
    Furthermore, if $\speed(p)\leq k_{n_0}$ and since $k_{n_2}\geq k_{n_0}+k_0\geq k_0+\speed(p)$, we see that $\comp^{\rG}_{v(p)}(p)$ intersects $B_{k_{n_2}}(G,o)$ in at least $k_0$ vertices.
    \smallskip
    
    On the event $E_n^c$, and since the edge weights of $\rG_n$ and $\rG$ are at a distance at most $1/(2k_n)$, we observe that the ordering of the edges on the ball of radius $k_{n_2}$ is the same between $\rG_n$ and $\rG$.
    In particular, as long as Prim's algorithm does not leave the ball of radius $k_{n_2}$, it explores the same vertices and in the same order in $\rG_n$ and in $\rG$.
    Using that, on the event $E_n^c$, $\speed(p)\leq k_{n_0}$, it follows that Prim's algorithm behaves in the same way on $\rG_n$ and on $\rG$, at least until it reaches $v(p)$.
    However, then, since $\comp^{\rG}_{v(p)}(p)$ intersects $B_{k_{n_2}}(\rG)$ on at least $k_0$ vertices and since no weight in $\rG$ larger than $p$ would be smaller than $p$ in $\rG_n$, it follows that $|\comp^{\rG_n}_{v(p)}(p)|\geq k_0$ and, by using that we are on $E_n^c$ again, it follows that $\comp^{\rG_n}_{v(p)}(p)=\comp^{\rG_n}_{(1)}(p)$.

    \def\tree{\mathcal{T}}
    Consider now the tree $\tree=\prim_{\speed(p)+\left|\comp^{\rG_n}_{(1)}(p)\right|}(\rG_n)$. We aim to prove that, on $E_n^c$,
    \begin{align*}
        \tree\equiv_r\rMSF^{\rG}_+(p)\,.
    \end{align*}
    For all $n$ such that $r\leq k_n$, since the weights on both balls are at least $1/2k_n$ away from each other thanks to the coupling between $\rG$ and $\rG_n$, it only suffices to check that the two graphs are isomorphic to each other until depth $r$.
    To do so, we first observe that, since Prim's algorithm reaches $\comp^{\rG_n}_{(1)}(p)$ after $\speed(p)+1$ steps, it follows that
    \begin{align*}
        \tree=\prim_{\speed(p)+1}(\rG_n)\cup\MST^{(\comp^{\rG_n}_{(1)}(p),w_n)}\,,
    \end{align*}
    \invisible{\RvdH{Is the $w_n$ here a $p$? Discuss. \em there was indeed a $p$ missing, but the $w_n$ is only here because the components are un-rooted and un-weighted.}}
    where we recall that $w_n=(w_n(e))_{e\in E(G_n)}$ are the edge weights on $G_n$, and we note that $\tree$ is the first part of Prim's algorithm combined with the minimum spanning tree on its largest component at level $p$.
    Using that
    \begin{align*}
        \MST^{(\comp^{\rG_n}_{(1)}(p),w_n)}=\MST^{\rG_n}\cap\comp^{\rG_n}_{(1)}(p),
    \end{align*}
    along with the properties of the coupling, it follows that
    \begin{align*}
        \tree\equiv_r\prim_{\speed(p)+1}(\rG)\cup\left(\MSF^{\rG}\cap\comp^{\rG}_{v(p)}(p)\right),
    \end{align*}
    as long as $k_n\geq r$ and we are on $E_n^c$.
    However, this last term is exactly the target structure thanks to Lemma~\ref{lem:eipc struct}.
    Thus, we have proved that the desired equivalence is true, which occurs with probability at least $1-\varepsilon$ for $n$ large enough.
    This is exactly the desired convergence statement.
\end{proof}

\subsection{Marginal convergence}\label{sec:marginals}

In this section, we use Proposition~\ref{prop:main result} to show that the multi-dimensional distributions converge as stated in Theorem~\ref{thm:Prim convergence}. We start with the one-dimensional distribution:

\begin{proposition}[One-dimensional convergence]\label{prop:one dim}
    Let Assumptions~\ref{ass:speed} and \ref{ass:local} hold.
    Further, assume that $\theta$ satisfies Assumption~\ref{ass:smooth}.
    Using the notations from the beginning of Section~\ref{sec:thm}, for any $r\geq1$ and $t\in[0,1]$,
    \begin{align*}
        \lim_{n\rightarrow\infty}\prob\Big(\prim_{k_n(t)}(\rG_n)\equiv_r\rMSF^{\rG}_+\big(\theta^{-1}(t)\big)\Big)=1\,.
    \end{align*}
\end{proposition}

\begin{proof}
    \def\epsdiv{5}
    Fix $\varepsilon>0$, $r\geq1$ and $t\in[0,1]$.
    Let $\delta>0$ be such that
    \begin{align}\label{eq:delta}
        \prob\Big(\inf\Big\{|\theta^{-1}(t)-w(e)|:e\in E\big(B_r(\rG)\big)\Big\}\leq\delta\Big)\leq\frac{\varepsilon}{\epsdiv}\,.
    \end{align}
    We start with the case $t=0$.

    When $t=0$, using the fact that $\rG_n\equiv_{k_n(0)}\rG$, it follows that
    \begin{align*}
        \prim_{k_n(0)}(\rG_n)\equiv_{k_n(0)}\prim_{k_n(0)}(\rG)\,.
    \end{align*}
    Then, since $k_n(0)\rightarrow\infty$, it also follows that
    \begin{align*}
        \lim_{n\rightarrow\infty}\prob\Big(\prim_{k_n(0)}(\rG)\equiv_r\prim_{\infty}(\rG)\Big)=1\,.
    \end{align*}
    It only remains to use that $\prim_\infty(\rG)=\rMSF^{\rG}_+(0)$, as stated in Lemma~\ref{lem:eipc struct}, to see that the desired result holds when $t=0$.

    Assume now that $t\in(0,1)$.
    Thanks to Assumption~\ref{ass:smooth} and the definition from \eqref{eq:theta inv}, $\theta^{-1}$ is continuous and strictly increasing from $[0,1)$ to $[p_c,1)$, with $\theta^{-1}(0)=p_c$, and we recall that $\theta^{-1}(1)=1$. Thus,
    we can find $p_-\leq\theta^{-1}(t)\leq p_+$ such that $\theta(p_-)<t<\theta(p_+)$ and $|p_+-\theta^{-1}(t)|,|p_--\theta^{-1}(t)|\leq\delta$.
    As a first step, we observe that $\theta(p_+)>t>0$, and therefore we can apply Proposition~\ref{prop:main result} to find $n_+$ such that, for all $n\geq n_+$,
    \begin{align*}
        \prob\Big(\prim_{\speed(p_+)+\left|\comp^{\rG_n}_{(1)}(p_+)\right|}(\rG_n)\equiv_r\rMSF^{\rG}_+(p_+)\Big)\geq1-\frac{\varepsilon}{\epsdiv}\,.
    \end{align*}
    Moreover, using that $|p_+-\theta^{-1}(p)|\leq\delta$ along with~\eqref{eq:delta}, the error in probability arising from replacing $p_+$ with $\theta^{-1}(p)$ in $\rMSF^{\rG}_+$ is $\varepsilon/\epsdiv$, leading to
    \begin{align*}
        \prob\Big(\prim_{\speed(p_+)+\left|\comp^{\rG_n}_{(1)}(p_+)\right|}(\rG_n)\equiv_r\rMSF^{\rG}_+\big(\theta^{-1}(p)\big)\Big)\geq1-\frac{2\varepsilon}{\epsdiv}\,.
    \end{align*}

    We next note that there are two possibilities:
    We either have that $\theta(p_-)>0$, in which case we can apply the same method and obtain that, for all $n\geq n_-$,
    \begin{align*}
        \prob\Big(\prim_{\speed(p_-)+\left|\comp^{\rG_n}_{(1)}(p_-)\right|}(\rG_n)\equiv_r\rMSF^{\rG}_+\big(\theta^{-1}(p)\big)\Big)\geq1-\frac{2\varepsilon}{\epsdiv}\,.
    \end{align*}
    Alternatively, $\theta(p_-)=0$, and, in that case, we use the definition of $k_n$ to see that
    \begin{align*}
        \prob\Big(\prim_{k_n(0)}(\rG_n)\equiv_r\prim_{\infty}(\rG)\Big)=1\geq1-\frac{\varepsilon}{\epsdiv}\,.
    \end{align*}
    Moreover, using Assumption~\ref{ass:speed} and Proposition~\ref{prop:speed}, we see that $\theta(p_-)=0$ implies that $\prob(\speed(p_-)=\infty)$, which, combined with Lemma~\ref{lem:eipc struct}, means that $\prim_\infty(\rG)=\rMSF^{\rG}_+(p_-)$ almost-surely.
    Again using the definition of $\delta$ and $p_-$, it follows that
    \begin{align*}
        \prob\Big(\prim_{k_n(0)}(\rG_n)\equiv_r\rMSF^{\rG}_+\big(\theta^{-1}(p)\big)\Big)\geq1-\frac{2\varepsilon}{\epsdiv}\,.
    \end{align*}
    Combing the previous results, we have proved that
    \begin{align*}
        \prob\Big(\prim_{L_n}(\rG_n)\equiv_r\rMSF^{\rG}_+\big(\theta^{-1}(p)\big)\equiv_r\prim_{U_n}(\rG_n)\Big)\geq1-\frac{4\varepsilon}{\epsdiv}\,,
    \end{align*}
    where
    \begin{align*}
        U_n=\speed(p_+)+\left|\comp^{\rG_n}_{(1)}(p_+)\right|,
    \end{align*}
    and
    \begin{align*}
        L_n&=\left\{\begin{array}{ll}
            \displaystyle\speed(p_-)+\left|\comp^{\rG_n}_{(1)}(p_-)\right| & \textrm{if $\theta(p_-)>0,$} \\
            k_n(0) & \textrm{otherwise}\,.
        \end{array}\right.
    \end{align*}
    It remains to prove that,  with high probability,
        \eqn{
        \label{bounds-U-and-L}
        L_n\leq k_n(t)\leq U_n.
        }
    \smallskip

    To show that the bounds in \eqref{bounds-U-and-L} hold, use Proposition~\ref{prop:theta}, as well as the fact that $\speed$ is almost-surely finite when $\theta$ is positive by Assumption~\ref{ass:speed} and Proposition~\ref{prop:speed}, to obtain that
    \begin{align*}
        \frac{U_n}{n}\probc\theta(p_+)>t,
    \end{align*}
    and
    \begin{align*}
        \frac{L_n}{n}\probc\left\{\begin{array}{ll}
            \displaystyle\theta(p_-)<t& \textrm{if $\theta(p_-)>0,$} \\
            0 & \textrm{otherwise.}
        \end{array}\right.
    \end{align*}
    Recalling that $k_n(t)/n\rightarrow t>0$, there exists $n_0\geq\max\{n_+,n_-\}$ such that, for all $n\geq n_0$,
    \begin{align*}
        \prob\Big(L_n\leq k_n(t)\leq U_n\Big)\geq1-\frac{\varepsilon}{\epsdiv}\,.
    \end{align*}
    Thus, for any $n\geq n_0$, and $k_n\geq r$,
    \begin{align*}
        \prob\Big(\prim_{k_n(t)}(\rG_n)\equiv_r\rMSF^{\rG}_+\big(\theta^{-1}(t)\big)\Big)\geq1-\varepsilon\,,
    \end{align*}
    where we have used that Prim's algorithm is a monotone process.
    This is exactly the desired result when $t\in(0,1)$. Thus, it only remains to cover the case $t=1$.
    However, this case uses the exact same proof as that of the case $0<t<1$, except that the coupling directly gives us that
    \begin{align*}
        \rMST^{\rG_n}=\prim_n(\rG_n)\equiv_r\rMSF^{\rG}_+\big(\theta^{-1}(1)\big)=\rMSF^{\rG},
    \end{align*}
    which we use as the upper bound instead of Prim's algorithm performed $U_n$ steps.
    Since the lower bound remains true in that case as well, the desired result follows.
\end{proof}

We have proved that the one-dimensional distribution converges with high probability to the desired limit, and extending this proof to the multi-dimensional distribution directly follows from the fact that the intersection of finitely many high-probability events is also a high-probability event:

\begin{corollary}[Muti-dimensional convergence]
\label{cor:marginal conv}
    Let Assumptions~\ref{ass:speed} and \ref{ass:local} hold.
   Further, assume that $\theta$ satisfies Assumption~\ref{ass:smooth}. 
    Using the notations from the beginning of Section~\ref{sec:thm}, for any $r\geq1$ and $t_,\ldots,t_k\in[0,1]$,
    \begin{align*}
        \lim_{n\rightarrow\infty}\prob\Big(\forall\ell\in[k]:\prim_{k_n(t_\ell)}(\rG_n)\equiv_r\rMSF^{\rG}_+\big(\theta^{-1}(t_\ell)\big)\Big)=1\,.
    \end{align*}
\end{corollary}

This concludes the section on marginal convergence, and in the next section, we use these results to prove that the process convergence occurs as stated in Theorem~\ref{thm:Prim convergence}.

\subsection{Process convergence}\label{sec:process}

In this section we conclude the first part of this work by providing the proof of Theorem~\ref{thm:Prim convergence}.
To do so, we simply verify that the three assumptions~\eqref{eq:lpc multi-dim},~\eqref{eq:lpc one-cont}, and~\eqref{eq:lpc small jump} from Section~\ref{sec:lpc} are satisfied.

\begin{proof}[Proof of Theorem~\ref{thm:Prim convergence}]
    Observe first that the convergence corresponding to~\eqref{eq:lpc multi-dim} directly follows from the definition of the coupling at the beginning of Section~\ref{sec:thm} along with Corollary~\ref{cor:marginal conv}.
    For the second condition, we observe that, if all edges in $B_r(\rG)$ have weight less than $1-\delta$, then $\rMSF^{\rG}_+(1-\delta)\equiv_r\rMSF^{\rG}_+(1)$.
    Using that $\rG$ is locally finite, this implies that, for all $r\geq1$,
    \begin{align*}
        \lim_{\delta\rightarrow0}\prob\left(d\Big(\rMSF^{\rG}_+(1-\delta),\rMSF^{\rG}_+(1)\Big)>\frac{1}{r+1}\right)=0\,,
    \end{align*}
    and this exactly corresponds to~\eqref{eq:lpc one-cont}.
    We now prove that the final condition in \eqref{eq:lpc small jump} holds.
    \smallskip

    \def\epsdiv{2}
    Fix $r\geq1$ and $\varepsilon>0$.
    Our goal is to show that there exists $n_0$ and $\delta_0>0$ such that, for all $n\geq n_0$ and $\delta\leq\delta_0$, we have
    \begin{align*}
        \prob\left(\exists t<t'<t''<t+\delta:\left\{\begin{array}{l}d\Big(\prim_{k_n(t)}(\rG_n),\prim_{k_n(t')}(\rG_n)\Big)>\frac{1}{r+1}\\d\Big(\prim_{k_n(t')}(\rG_n),\prim_{k_n(t'')}(\rG_n)\Big)>\frac{1}{r+1}\end{array}\right.\right)\leq\varepsilon\,.
    \end{align*}
    Using first that $\theta^{-1}$ is continuous and strictly monotone from $[0,1]$ to $[p_c,1]$, as implied by Assumption~\ref{ass:smooth}, we observe that
    \begin{align*}
        \Delta(\theta,\delta)=\sup\Big\{\big|\theta^{-1}(t+\delta)-\theta^{-1}(t)\big|:t\in[0,1-\delta]\Big\}
    \end{align*}
    is decreasing in $\delta$, strictly positive whenever $\delta>0$, and converges to $0$ as $\delta\rightarrow0$.
    Combining this with the fact that $\rG$ is locally finite, we can set $\delta_0$ so that
    \begin{align*}
        \prob\left(\min\Big\{\big|w(e)-w(e')\big|:e\neq e\in E\big(B_r(\rG)\big)\Big\}\leq\Delta(\theta,3\delta_0)\right)\leq\frac{\varepsilon}{\epsdiv}\,.
    \end{align*}
    We observe that the complement of the previous event implies that no segment of length less than $\Delta(\theta,2\delta_0)$ contains two edge weights.
    From this observation, the previous definition of $\delta_0$ implies that
    \begin{align*}
        \prob\Big(\exists t,\exists e\neq e'\in E\big(B_r(\rG)\big):\theta^{-1}(t)\leq w(e),w(e')\leq \theta^{-1}(t+2\delta_0)\Big)\leq\frac{\varepsilon}{\epsdiv}\,.
    \end{align*}
    Consider now a sequence of closed intervals $([t^-_\ell,t^+_\ell])_{1\leq\ell\leq k}$ of length $2\delta_0$ such that any interval of length $\delta_0$ is fully included into one of these intervals.
    One can for example choose $t^-_\ell=(\ell-1)\delta_0$ with $t^+_\ell=t^-_\ell+2\delta_0$ for $\ell\leq1/\delta-1$ and add the extra interval $[1-2\delta_0,1]$ at the end.
    Using Corollary~\ref{cor:marginal conv}, we can find $n_0$ such that, for all $n\geq n_0$,
    \begin{align*}
        \prob\Big(\forall\ell\in[k]\colon \prim_{k_n(t_\ell^\pm)}(\rG_n)\equiv_r\rMSF^{\rG}_+\big(\theta^{-1}(t_\ell^\pm)\big)\Big)\geq1-\frac{\varepsilon}{\epsdiv}\,.
    \end{align*}
    We conclude this proof by showing that $\delta_0$ and $n_0$ satisfy the desired property.
    We assume without loss of generality that $n_0$ is large enough so that $\rG_n\equiv_r\rG$ for all $n\geq n_0$.
    \smallskip

    Consider $n\geq n_0$ and $\delta\leq\delta_0$.
    Consider three times $t<t'<t''<t+\delta$.
    Then $t,t',t''$ are all contained in an interval of length $\delta_0$ and so they all belong to one of the interval $[t^-_\ell,t^+_\ell]$ for some $\ell$.
    However, now, if we assume that $\prim_{k_n(t_\ell^\pm)}(\rG_n)=\rMSF^{\rG}_+(\theta^{-1}(t_\ell^\pm))$ and that there exists no pair of edges whose difference of weights is less than $\Delta(\theta,2\delta_0)$, the existing coupling between $\rG_n$ and $\rG$ as well as the monotony of Prim's algorithm implies that, for all $\Tilde{t}\in\{t,t',t''\}$,
    \begin{align*}
        \prim_{k_n(\Tilde{t})}(\rG_n)\equiv_r\rMSF^{\rG}_+(t_\ell^-),
    \end{align*}
    or
    \begin{align*}
        \prim_{k_n(\Tilde{t})}(\rG_n)\equiv_r\rMSF^{\rG}_+(t_\ell^+)\,.
    \end{align*}
    It follows that
    \begin{align*}
        \min\bigg\{d\Big(\prim_{k_n(t)}(\rG_n),\prim_{k_n(t')}(\rG_n)\Big),d\Big(\prim_{k_n(t')}(\rG_n),\prim_{k_n(t'')}(\rG_n)\Big)\bigg\}\leq\frac{1}{r+1},
    \end{align*}
    and thus
    \begin{align*}
        &\prob\left(\exists t<t'<t''<t+\delta:\left\{\begin{array}{l}d\Big(\prim_{k_n(t)}(\rG_n),\prim_{k_n(t')}(\rG_n)\Big)>\frac{1}{r+1}\\d\Big(\prim_{k_n(t')}(\rG_n),\prim_{k_n(t'')}(\rG_n)\Big)>\frac{1}{r+1}\end{array}\right.\right)\\
        &\pskip\leq\prob\Big(\exists\ell\in[k]:\prim_{k_n(t_\ell^\pm)}(\rG_n)\not\equiv_r\rMSF^{\rG}_+\big(\theta^{-1}(t_\ell^\pm)\big)\Big)\\
        &\pskip\pskip+\prob\Big(\exists t,\exists e\neq e'\in E\big(B_r(\rG)\big):\theta^{-1}(t)\leq w(e),w(e')\leq \theta^{-1}(t+2\delta_0)\Big)\leq\varepsilon\,,
    \end{align*}
    as desired.
    This proves that the third condition in \eqref{eq:lpc small jump}, as stated in Section~\ref{sec:lpc}, is satisfied and thus concludes the proof of Theorem~\ref{thm:Prim convergence}.
\end{proof}


\section{Examples and counter-examples}\label{sec:examples}

In Section~\ref{sec-examples}, we first show that Theorem~\ref{thm:Prim convergence} applies to four common graph sequences and limits: the multi-dimensional grid, the Erd\H{o}s-R\'enyi graph, the configuration model, and the preferential attachment graph.
In Section~\ref{sec-assumptions}, we then provide two interesting examples of graphs satisfying exactly one of Assumption~\ref{ass:speed} and Assumption~\ref{ass:local} and state the conjecture that Assumption~\ref{ass:speed} is not strictly necessary.

\subsection{Example graph sequences}
\label{sec-examples}
In this section, we discuss important examples of graph sequences to which our results apply.
\medskip

\paragraph{\bf Grids} The minimal spanning tree on grids is closely related to percolation on $\mathbb{Z}^d$, which is a topic of intense research. See \cite{Grimmett1999} for overviews of the field, and \cite{HeyHof17} for an overview of percolation in high dimensions. Assumptions \ref{ass:speed}, \ref{ass:smooth} and \ref{ass:local} are all well known. In particular, Assumption \ref{ass:speed} holds since the infinite component is {\em unique} (see \cite[Theorem 8.1]{Grimmett1999}), so that the minimal spanning forest (and even Prim's algorithm) needs to enter it at some point (and then in fact will never leave it). Assumption \ref{ass:smooth} is \cite[Theorem 8.8]{Grimmett1999}. Finally, Assumption \ref{ass:local} follows from the fact that $|\comp^{\rG_n}_{(1)}(p)|/|V(G_n)|\convp \theta(p)$, while $|\comp^{\rG_n}_{(2)}(p)|/|V(G_n)|\convp 0$ (see e.g., \cite{HofRed06}), and the fact that this is equivalent to Assumption \ref{ass:local} by \cite{VanderHofstad2021b}. Thus, our results apply to this important case.
\medskip

\paragraph{\bf  Complete graph} Percolation on the complete graph is the Erd\H{o}s-R\'enyi random graph, and the connections between the minimal spanning tree and the critical Erd\H{o}s-R\'enyi random graph are tight \cite{Addario-Berry2017}. Our results apply, see also below for the discussion on the configuration model. The setting is slightly different though, as we now need to take $p=\lambda/n$, and consider $\lambda$ as the parameter instead of $p\in[0,1]$. 
\medskip

\paragraph{\bf Configuration model}
Assumption~\ref{ass:speed} is true as the local limit is a unimodular branching-process tree. Assumption \ref{ass:local} is true by \cite[Section 4.3.1]{VanderHofstad2024}, while continuity of $\theta$, which implies Assumption \ref{ass:smooth}, follows from \cite{Michelen2020}. The above references can easily be extended to the Erd\H{o}s-R\'enyi random graph.
\medskip

\paragraph{\bf Preferential attachment models}
Again our results apply. Assumption~\ref{ass:speed} is true by \cite{HazHofRay23}, which also proves the continuity of $p\mapsto \theta(p)$ in  Assumption \ref{ass:smooth}. Assumption \ref{ass:local} again holds since $|\comp^{\rG_n}_{(1)}(p)|/|V(G_n)|\convp \theta(p)$, while $|\comp^{\rG_n}_{(2)}(p)|/|V(G_n)|\convp 0$ by \cite{HazHofRay23}, and the fact that this is equivalent to Assumption \ref{ass:local} by \cite{VanderHofstad2021b}.
\smallskip

We conclude that there is a wide range of models to which our results apply, and where the dynamic local limit of Prim's algorithm gives insight into how the minimal spanning tree is formed.

\subsection{A discussion on the assumptions}
\label{sec-assumptions}

In this section, we provide some examples of graph sequences which do not satisfy at least one of the assumptions and study the behaviour of Prim's algorithm in these cases.
\smallskip 

\paragraph{\bf Union graphs} For the rest of the section, we let $(\rG^1_n)_{n\geq1}$ and $(\rG^2_n)_{n\geq1}$ be two sequences of graphs locally converging to $\rG^1$ and $\rG^2$, with percolation functions $\theta^1$ and $\theta^2$ and both satisfying all the assumptions from Theorem~\ref{thm:Prim convergence}.
One can for example consider $(\rG^1_n)_{n\geq1}$ and $(\rG^2_n)_{n\geq1}$ to be sequences of random $d_1$- and $d_2$-regular graphs.
We further assume that both $\rG^1_n$ and $\rG^2_n$ have size $n$ and let $(\rG_n)_{n\geq1}$ be the graph obtained by taking the union of $\rG^1_{\lfloor n/2\rfloor}$ and $\rG^2_{\lceil n/2\rceil}$, so that $\rG_n$ also has size $n$. We call such graphs {\em union graphs}. They are of interest since they are the simplest examples where we do have local convergence, but our assumptions may be false.
It may be convenient to assume that the two graphs are connected by adding an arbitrary edge between the two components (or even an arbitrary, but finite, number of edges), and we note that this does not affect the local limit.
\smallskip

Before observing the effect of this construction on the different assumptions, we first observe that the local limit of $(\rG_n)_{n\geq1}$ is the graph $\rG$ which is distributed as $\rG^1$ with probability $1/2$, and as $\rG^2$ otherwise. This implies that $\theta=\theta^{\rG}=(\theta^1+\theta^2)/2$.

\medskip

\paragraph{\bf Satisfying the assumptions for union graphs}

To start this section, we observe in which cases $(\rG_n)_{n\geq 1}$ and its limit satisfy the different assumptions.

If we denote by $p_c^1$ and $p_c^2$ the critical percolation for $\rG^1$ and $\rG^2$, we observe that Assumption~\ref{ass:speed} is satisfied if and only if $p_c^1=p_c^2$.
Indeed, by conditioning the probability from Assumption~\ref{ass:speed} according to whether $\rG$ is distributed as $\rG^1$ or $\rG^2$, we observe that
\begin{align*}
    \prob\Big(\exists v\in V(\rMSF^{\rG}):\big|\comp^{\rG}_v(p)\big|=\infty\Big)&=\frac{1}{2}\Big(\mathbbm{1}_{\{\theta^1(p)>0\}}+\mathbbm{1}_{\{\theta^2(p)>0\}}\Big)\,.
\end{align*}
From this formula, and using that $\theta=(\theta^1+\theta^2)/2$, we see that Assumption~\ref{ass:speed} holds if and only if $\theta^1$ and $\theta^2$ become positive at the same time, i.e., if and only if $p_c^1=p_c^2$.
\smallskip

To test whether Assumption~\ref{ass:smooth} holds here, we first observe that, if $\theta^1$ and $\theta^2$ both satisfy this assumption, then $\theta$ also satisfies it when $p_c^1=p_c^2$. In the case where $p_c^1\neq p_c^2$, since $p_c=\min\{p_c^1,p_c^2\}$, we see that Assumption~\ref{ass:smooth} is satisfied if and only if the graph with larger critical value has a {\em continuous} percolation function.
For example, if we assume that $p_c^1<p_c^2$, then we want $\theta^2$ to be continuous on the whole interval $[0,1]$ (or equivalently on $(p_c^1,1]$ since $\theta^2(p)=0$ whenever $p<p_c^2$).
\smallskip

Finally, we can see that Assumption~\ref{ass:local} does not hold if both $\theta^1(p)$ and $\theta^2(p)$ are positive for some $p<1$.
Indeed, in that case both $\rG^1_{\lfloor n/2\rfloor}$ and $\rG^1_{\lceil n/2\rceil}$ have components of linear sizes, and using that only one edge connects the two graphs, we see that, for any $k\geq0$,
\begin{align*}
    &\liminf_{n\rightarrow\infty}\frac{1}{n^2}\expec\bigg[\Big|\Big\{u,v\in V(\rG_n)\colon \big|\comp^{\rG_n}_u(p)\big|\geq k,\big|\comp^{\rG_n}_v(p)\big|\geq k,\comp^{\rG_n}_u(p)\cap\comp^{\rG_n}_v(p)=\varnothing\Big\}\Big|\bigg]\\
    &\qquad\geq(1-p)\frac{\theta^1(p)\theta^2(p)}{4}\,.
\end{align*}
We further observe that Assumption~\ref{ass:local} is still not satisfied when we remove the edge connecting the two graphs, or add a finite number more of such edges.

\medskip

\paragraph{\bf Satisfying our theorem for union graphs}

In this part, we now evaluate in which case the sequence of union graphs $(\rG_n)_{n\geq1}$ satisfies Theorem~\ref{thm:Prim convergence}.
Interestingly, while it does not always satisfy it in its current form, it actually does by adapting the formula, leading us to believe that a more general result can be extracted from these observations.

First of all, since the result applies to $(\rG_n)_{n\geq1}$ when all assumptions are met, we now focus on the cases where one of the assumptions is {\em not} satisfied.
We thus assume without loss of generality that $0<p_c^1<p_c^2<1$.
In that case, if the root of $\rG_n$ lies within the half of the vertices corresponding to $\rG^i$, then, on any ball of finite radius, it seems that $\prim_k(\rG_n)$ behaves similarly to $\prim_k(\rG^i_{\lfloor n/2\rfloor})$, whose local limit is $\rMSF^{\rG^i}_+(p)$ with $p$ satisfying an equation of the form $\theta^i(p)\simeq2k/n$.

The exact behaviour is actually a lot more complicated if there is an edge connecting the two graphs.
Indeed, depending on when the endpoint of this edge is explored, and what the edge weight is, the exploration of the current graph $\rG^i_m$ might be suddenly interrupted in order to explore the other graph.
Naturally, this behaviour becomes more and more complicated as we add more and more edges between the two graphs, while keeping the local limit intact.

For simplicity, we now assume that the two graphs are actually not connected to each other.
Forgetting about the process convergence for now, the previous observations lead, in that case, to the following convergence:
\begin{align*}
    \prim_{k_n(t)}(\rG_n)\rightarrow\left\{\begin{array}{ll}
        \rMSF^{\rG^1}_+\Big((\theta^1)^{-1}\big(2t\big)\Big) & \textrm{w.p. }1/2 \\
        \rMSF^{\rG^2}_+\Big((\theta^2)^{-1}\big(2t\big)\Big) & \textrm{w.p. }1/2\,,
    \end{array}\right.
\end{align*}
meaning that Theorem~\ref{thm:Prim convergence} can still hold without all the assumptions being satisfied, by adapting the statement.
Note that a similar formula can be given for more complex union graphs, where we combine more than two graphs and do not necessarily have equal proportions of vertices in each graph.

While the previous convergence already has a major caveat in the fact that we might completely change the previous behaviour by simply adding one edge beween the two graphs, it still gives the impression that Theorem~\ref{thm:Prim convergence} can easily be extended to more complex graph structures.
However, a very similar case as the previous example, but with a distinctly different behaviour, is the sequence of graphs $(\rG_n)_{n\geq1}$ where instead of taking the union between two graphs, we flip a (fair) coin and choose the sequence accordingly.
More precisely, $(\rG_n)_{n\geq1}$ is exactly $(\rG_n^1)_{n\geq1}$ with probability $1/2$, and $(\rG^2_n)_{n\geq1}$ otherwise.
Observe that this sequence of random graphs has the same limit as the union graph: it is $\rG^1$ with probability $1/2$ and $\rG^2$ otherwise.
We emphasize, though, that the behaviour of Prim's algorithm is quite different, as the exploration will have to go through the whole graph and thus look more like
\begin{align*}
    \prim_{k_n(t)}(\rG_n)\rightarrow\left\{\begin{array}{ll}
        \rMSF^{\rG^1}_+\Big((\theta^1)^{-1}(t)\Big) & \textrm{w.p. }1/2, \\
        \rMSF^{\rG^2}_+\Big((\theta^2)^{-1}(t)\Big) & \textrm{w.p. }1/2\,.
    \end{array}\right.
\end{align*}
Once again, in this case it is possible to extend Theorem~\ref{thm:Prim convergence} without satisfying all the assumptions.
\medskip

Because of the previous examples, we believe that Theorem~\ref{thm:Prim convergence} can be extended past the assumptions we have made here.
However, the complexity of possible behaviours, as well as the strong dependency on the {\em exact details} of the construction sequence of graphs considered, and not only their local limit, made us prefer its current format.
It is still worth noting that, while it is possible to create much more intricate local behaviour than taking unions of converging graphs, or choosing sequence of converging graphs at random, these two interesting cases can actually still be well studied with our main result by simply adapting the limits accordingly.
\bigskip

\noindent
\begin{minipage}{12.5cm}
{\paragraph{\bf Acknowledgement}
This work is supported by the Netherlands Organisation for Scientific Research (NWO) through Gravitation-grant NETWORKS-024.002.003.
The work of RvdH is further supported by the National Science Foundation under Grant No. DMS-1928930 while he was in residence at the Simons Laufer Mathematical Sciences Institute in Berkeley, California, during the spring semester 2025. The work of BC is further supported by the European Union's Horizon 2020 Research and Innovation Programme under the Marie Sk\l{}odowska-Curie grant agreement No.~101034253.
}
\end{minipage}
  \hfill
  \begin{minipage}{2.5cm}
    \vspace{-4pt}
    \raggedleft
    \includegraphics[width = 2.5cm]{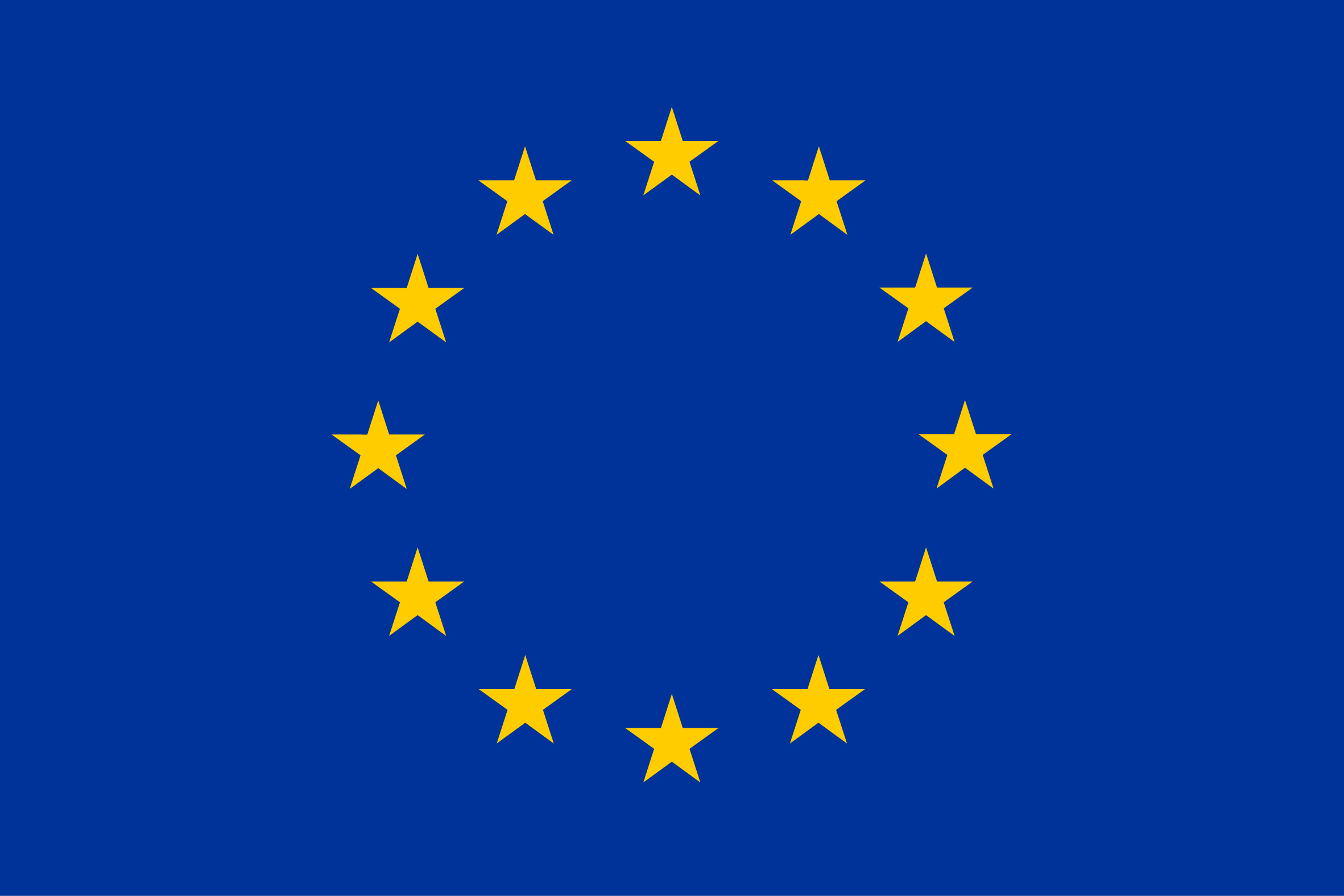}
  \end{minipage}
\vspace{0.5cm}

\appendix
\section{List of symbols and notation}\label{app:notations}
\begin{longtable}{ll}
    MST & Minimum spanning tree\\
    IPC & Invasion percolation cluster\\
    &\\
w.h.p. & With high probability \\
i.i.d. & Independent and identically distributed\\
$\uniform$&Standard uniform distribution on $[0,1]$\\
&\\
    $G=(V(G),E(G))$ & Graph\\
    $\wG=(G,w)$& Weighted graph\\
    $(G,o)$& Rooted graph with root $o$\\
    $\rG=(G,o,w)$ & Rooted weighted graph with i.i.d.\ $\uniform$ edge weights\\
    $\wG(p)$& Graph $\wG$ obtained by only keeping edges of weight less than $p$\\
    $\rG(p)$ & Rooted weighted graph obtained by only keeping edges of weight less than $p$\\
    $\MSF^{\rG}=\MSF^{(G,w)}$ &Minimum spanning forest of $\wG$\\
    $\rMSF^{\rG}=(\MSF^{\rG},o)$& Rooted minimum spanning forest of $\rG$\\
    $\prim_{k}(\rG)$ & Rooted subtree of $\rG$ obtained by performing $k$ steps of Prim's algorithm\\
    $\rMSF^{\rG}_+(p)$ & Expanded invasion percolation cluster at level $p$ (see Section~\ref{sec:msf + eipc})\\   
   
    &\\
        $\theta(p)$ &Percolation survival probability\\
        $p_c$ & Critical percolation coefficient\\
        $\comp^{\rG}_o(p)$&Component of $o$ when only keeping edges with weight less than $p$\\
        $\comp^{\rG_n}_{(1)}(p)$ &Largest component when only keeping edges with weight less than $p$\\
        $\comp^{\rG_n}_{(2)}(p)$ &Second largest component when only keeping edges with weight less than $p$\\
        &\\
$\probc $ & Convergence in probability\\
$\lwc $ & Local weak convergence (see Section~\ref{sec:lwc})\\
$\lcp $ & Local convergence in probability (see Section~\ref{sec:lwc})\\
$\lpc $ & Local process convergence (see Section~\ref{sec:lpc})\\
\end{longtable}
\bibliographystyle{abbrv}
\bibliography{prim_conv}

\end{document}